\documentclass[10pt]{amsart}
\usepackage{amsfonts}
\usepackage{amsmath}
\usepackage{amssymb}
\usepackage{graphicx}
\usepackage{geometry}
\usepackage{verbatim}
\usepackage{hyperref}
\usepackage{charter,eulervm}
\usepackage{csquotes}

\swapnumbers
\theoremstyle{plain}
\newtheorem{theorem}{Theorem}[subsection]
\newtheorem{corollary}[theorem]{Corollary}
\newtheorem{proposition}[theorem]{Proposition}

\newtheorem{lemma}[theorem]{Lemma}

\theoremstyle{definition}
\newtheorem{definition}[theorem]{Definition}

\newtheorem{example}[theorem]{Example}

\newtheorem{question}[theorem]{Question}

\theoremstyle{remark}
\newtheorem{observations}[theorem]{Observations}
\newtheorem{remarks}[theorem]{Remarks}

\numberwithin{equation}{subsection}
\DeclareMathOperator{\coz}{coz}

\def\downset#1{\left\downarrow{#1}\right\downarrow}

\begin{document}
\title[Pointfree Pointwise Suprema ]{Pointfree Pointwise Suprema \\
in Unital Archimedean $\ell$-Groups}
\author{Richard N. Ball, Anthony W. Hager, and Joanne Walters-Wayland}
\address[Ball]{Department of Mathematics, University of Denver, Denver, CO
80208, U.S.A.}
\address[Hager]{Department of Mathematics, Wesleyan University, Middletown,
Connecticut, 06457}
\address[Walters-Wayland]{Center for Excellence in Computation, Algebra, and
Topology, Chapman University, Orange, California}
\date{\today}
\thanks{File name: Pointfree Pointwise Suprema.tex}
\keywords{archimedean $\ell$-group, extremally disconnected space, 
$\sigma$-complete $\mathbf{W}$-object, completely regular $P$-frame}
\subjclass[2010]{06D22,06F20;54D15,54A05}
\begin{abstract}
We generalize the concept of the pointwise supremum of real-valued functions
to the pointfree setting. The concept itself admits a direct and intuitive
formulation which makes no mention of points. But our aim here is to
investigate pointwise suprema of subsets of $\mathcal{R}L$, the family of
continuous real valued functions on a locale, or pointfree space.

Thus our setting is the category $\mathbf{W}$ of archimedean lattice-ordered
groups ($\ell$-groups) with designated weak order unit, with morphisms which
preserve the group and lattice operations and take units to units. This is
an appropriate context for this investigation because every 
$\mathbf{W}$-object can be canonically represented as a subobject of 
some $\mathcal{R}L$.

We show that the suprema which are pointwise in the Madden representation
can be characterized purely algebraically.  They are precisely the suprema which 
are context-free, in the sense of being preserved in every $\mathbf{W}$ 
homomorphism out of of $G$. We show that closure under such suprema 
characterizes the $\mathbf{W}$-kernels among the convex $\ell$-subgroups. 
And we prove that all existing (countable) joins in $\mathcal{R}L$ are pointwise
iff $L$ is boolean (a $P$-frame).   

This leads up to the appropriate analog of the Nakano-Stone Theorem: a (completely
regular) locale $L$ has the feature that $\mathcal{R}L$ is conditionally
pointwise complete ($\sigma $-complete), i.e., every bounded (countable)
family from $\mathcal{R}L$ has a pointwise supremum in $\mathcal{R}L$, 
iff $L $ is boolean (a $P$-locale).

It is perhaps surprising that pointwise suprema can be characterized purely
algebraically, without reference to a representation. They are the
context-free suprema, in the sense that the pointwise suprema are precisely
those which are preserved by all morphisms out of $G$. We adopt the latter
attribute as the final, representation-free definition of pointwise suprema.

Thus emboldened, we adopt a maximally broad definition of unconditional
pointwise completeness ($\sigma$-completeness): a divisible $\mathbf{W}$%
-object $G$ is pointwise complete ($\sigma$-complete) if it contains a
pointwise supremum for every subset which has a supremum in any extension.
We show that the pointwise complete ($\sigma$-complete) $\mathbf{W}$-objects
are those of the form $\mathcal{R}L$ for $L$ a boolean locale ($P$-locale).
Finally, we show that a $\mathbf{W}$-object $G$ is pointwise $\sigma$%
-complete iff it is epicomplete.
\end{abstract}

\maketitle
\tableofcontents

\section{Introduction}

When considering the suprema of real-valued functions, it is often relevant
to know whether this supremum coincides with the function obtained by taking
the supremum of the real values at each point. Here we propose a natural
generalization of this notion to the pointfree setting. We first define
pointwise suprema in $\mathcal{R}L$, where the classical definition can be
naturally articulated. We then show that this concept is actually
independent of the representation of a particular $\mathbf{W}$-object as a
subobject of $\mathcal{R}L$. For it is precisely the pointwise suprema in $%
\mathcal{R}L$ which are preserved by every $\mathbf{W}$-morphism out of $%
\mathcal{R}L$. We take advantage of this unexpected information by adopting
the latter attribute as the final, purely algebraic definition of pointwise
supremum: an element $g \in G$ is the pointwise join of a subset $K
\subseteq G^+$ iff $\theta(g) = \bigvee \theta [K]$ for all $\mathbf{W}$%
-morphisms $\theta $ out of $G$.

The notion of pointfree pointwise suprema has several useful applications. 
For example, a convex $\ell$ subgroup $K$ of a 
$\mathbf{W}$-object $G$ is a $\mathbf{W}$-kernel iff it is pointwise closed, 
i.e., iff $K$ contains any pointwise join of a subset of $K$ which exists in $G$.
And all existing (countable) suprema in $\mathcal{R}L$ are pointwise iff 
$L$ is boolean (a $P$-frame). 

This leads directly to Nakano-Stone type theorems. One of our main results is
Theorem \ref{Thm:3}: $\mathcal{R}L$ is conditionally pointwise complete 
($\sigma $-complete) iff $L$ is boolean (a $P$-frame).

Unconditional pointwise completeness requires that certain unbounded subsets
of a given $\mathbf{W}$-object $G$ have pointwise suprema in $G$, but, of
course, not all subsets can have suprema in $G$. The most permissive
criterion for a subset to have a pointwise supremum in $G$ is that the
subset have a supremum in some extension of $G$. We adopt this criterion as
our definition of unconditional pointwise completeness in Section~\ref{Sec:6}, 
and then show that the $\mathbf{W}$-objects which enjoy this attribute are
precisely those of the form $\mathcal{R}L$ for $L$ a Boolean frame, and
those which enjoy the corresponding unconditional $\sigma$-completeness are
those of the form $\mathcal{R}L$ for $L$ a $P$-frame. Finally, we show that the 
pointwise complete ($\sigma$-complete) objects form a full bireflective subcategory 
of $\mathbf{W}$.

The paper is organized as follows. After a preliminary Section \ref{Sec:3},
we briefly outline Madden's pointfree representation for $\mathbf{W}$ in
Section \ref{Sec:2}. We define pointwise suprema in Section \ref{Sec:4},
first in $\mathcal{R}L$ in Subsection \ref{Subsec:5} and then in $\mathbf{W}$
in Subsection \ref{Subsec:6}. In Section \ref{Sec:8} we give several different
applications of the notion of pointwise suprema.  We show that the 
$\mathbf{W}$-kernels of a particular $\mathbf{W}$-object are precisely
the pointwise closed convex $\ell$-subgroups. We show that all existing 
(countable) suprema in $\mathcal{R}L$ are pointwise iff $L$ is 
boolean (a $P$-frame). And we show that any element of a $\mathbf{W}$-object
is the pointwise join of its truncates, we characterize the sequences that 
can be realized as the truncates of an element of an extension, and we 
show that every such sequence has a pointwise supremum in $\mathcal{RM}G$. 
(Notation to be introduced subsequently.)  

Section \ref{Sec:5} is devoted to conditional pointwise completeness; the
main result here is the pointfree Nakano-Stone Theorem for conditional
pointwise completeness, Theorem \ref{Thm:3}. This result makes heavy use of
the pointfree generalization of the classical theorem, a beautiful result of
Banaschewski and Hong (\cite{BanaschewskiHong:2003}) which appears here in
embellished form as Theorem \ref{Thm:1}.

Section \ref{Sec:6} takes up unconditional pointwise completeness. A reveiw
of the well known facts concerning essential extensions constitutes
Subsection \ref{Subsec:7}, and a review of the less well known facts
concerning cuts occupies Subsection \ref{Subsec:8}. The section culminates
in Subsection \ref{Subsec:9}, in which we summarize our findings as they
pertain to unconditional pointwise completeness in Theorem \ref{Thm:4}.

\section{Preliminaries\label{Sec:3}}

Our notation and terminology is conventional for the most part, save only
for our notation for downsets. A subset $K$ of a poset $G$ is a 
\emph{downset} if $g \leq k \in K$ implies $g \in K$. We write 
\begin{equation*}
\downset{K}_G \equiv  \left\{ g \in G : g \leq k 
\text{ for some $k \in K$}\right\} 
\end{equation*}
for the downset in $G$ generated by a subset $K \subseteq G$. We drop the
subscript whenever it is unambiguous to do so. Upsets are defined and
denoted dually.

For a $\mathbf{W}$-object $G$, we denote by $\mathbb{R}^+(G)$
the set of those positive real numbers for which the corresponding
constant function is present in $G$.  Thus to say that 
$\bigwedge \mathbb{R}^+(G) = 0$ is to say that 
$G$ contains arbitrarily small positive multiples of $1$.  This is a weakening of 
the condition of being divisible which plays a prominent role in our
results. 

Good general references are \cite{AndersonFeil:1987} and \cite{Darnel:1994}
for $\ell$-groups, \cite{GillmanJerison:1960} for $\mathcal{C}X$, \cite%
{HagerRobertson:1977} for an introduction to $\mathbf{W}$, \cite%
{MaddenVermeer:1986} and \cite{BallHager:1991} for the pointfree, or Madden
representation for $\mathbf{W}$, \cite{Johnstone:1982} and \cite%
{PicadoPultr:2012} for general frame theory, and the many papers of Bernhard
Banaschewski, the tireless fount of knowledge of pointfree topology.

In spite of our use of the localic terminology in the abstract and
introduction, we prefer the algebraic language of frames and frame
morphisms. Henceforth, $\mathcal{R}L$ stands for the $\mathbf{W}$-object of
frame maps $g : \mathcal{O}\mathbb{R} \to L$, where $\mathcal{O}\mathbb{R}$
is the frame of open subsets of the real numbers $\mathbb{R}$ and $L$ is a
frame, assumed completely regular unless otherwise explicitly stipulated.

\section{A brief synopsis of the Madden representation\label{Sec:2}}

We mention here some of the technical results, familiarity with which will
be assumed in the sequel. The reader may skip this section upon a first
reading, returning to it as necessary.

\subsection{Calculation in $\mathcal{R}L$}

The arithmetic operations on $\mathbb{R}$ beget corresponding operations on $%
\mathcal{R}L$ as follows. We write $\vec{f}$ for $(f_1,f_2,\dots,f_n) \in (%
\mathcal{R}L)^n$, $\vec{U}$ for $(U_1,U_2,\dots,U_n) \in (\mathcal{O}\mathbb{%
R})^n$. For a continuous function $w:\mathbb{R}^{n}\rightarrow\mathbb{R}$ we
write $w( \vec{U}) \subseteq U$ to mean $U_{1}\times U_{2}\times\cdots\times
U_{n}\subseteq w^{-1}\left(U\right) $.

\begin{theorem}[{\protect\cite[3.1.1]{BallWalters:2002}}]\label{Thm:2}
The canonical lifting 
$w^{\prime}:\left( \mathcal{R}L\right)^{n}\rightarrow\mathcal{R}L$ 
of a continuous function $w:\mathbb{R}^{n}\rightarrow\mathbb{R}$ 
is given by the formula
\begin{equation*}
w^{\prime}( \vec{f}) \left( U\right) 
= \bigvee\limits_{w( \vec{U}) \subseteq U}
     \bigwedge\limits_{1\leq i\leq n} f_{i}\left( U_{i}\right), 
      \quad f_{i}\in \mathcal{R}L,\quad U\in\mathcal{O}\mathbb{R}.
\end{equation*}
The $U_{i}$'s in the supremum range over $\mathcal{O}\mathbb{R}$, and may be
taken to be rational intervals.
\end{theorem}

The formula also applies to constant functions; the frame map lifted from
the constant function $x\longmapsto r$ is given by%
\begin{equation*}
r\left( U\right) =\left\{ 
\begin{array}{cc}
\top & \text{if }r\in U \\ 
\bot & \text{if }r\notin U%
\end{array}
\right. .
\end{equation*}

Theorem \ref{Thm:2} provides a ready proof of a special case of Weinberg\rq{}%
s Theorem (\cite{Weinberg:1963}). A term is an expression built up from
variables and constants using the operations $+$, $-$, $\vee$, and $\wedge$.
An identity is an equation with terms on either side. Weinberg\rq{}s Theorem
asserts that an equation holds in $\mathbb{R}$ iff it holds in every abelian 
$\ell$-group.

\begin{corollary}
\label{Cor:2}Any identity which holds in $\mathbb{R}$ also holds in any $%
\mathcal{R}L$.
\end{corollary}

\begin{proof}
The terms on either side of the identity determine two functions $w_{i} :%
\mathbb{R}^{n}\rightarrow\mathbb{R}$, and these functions coincide because
the identity holds in $\mathbb{R}$. Therefore the liftings $w_{i}^{\prime}$
of these functions to $\mathcal{R}L$ coincide by Theorem \ref{Thm:2}.
\end{proof}

\subsection{A few useful formulas}

We record here a small number of formulas which will be especially useful in
what follows. They may be derived using from Theorem \ref{Thm:2} or even
Corollary \ref{Cor:2}. Details can be found in the literature by following
the references.

Lemma \ref{Lem:3} implies that a frame map $f:\mathcal{O}\mathbb{R}%
\rightarrow L$ is completely determined by its values on the right rays. For 
$f$ is clearly determined by its values on the base for $\mathcal{O}\mathbb{R%
}$ consisting of the open intervals $\left( r,s\right) $, $r<s$, and $%
f\left( r,s\right) =f\left( -\infty,s\right) \wedge f\left( r,\infty\right) $%
, and the left ray $f(-\infty,s)$ can be expressed in terms of the right
rays using the pseudocomplementation operator in the frame: 
\begin{equation*}
a^* \equiv \bigvee_{a \wedge b = \bot} b. 
\end{equation*}

\begin{lemma}[{\protect\cite[3.1.1]{BallHager:1991}}]\label{Lem:3}
For any $f,g\in\mathcal{R}L$ and $r\in\mathbb{R}$,

\begin{enumerate}
\item 
$f\left( -\infty,r\right) 
=\bigvee\limits_{s<r}f\left( s,\infty\right) ^{\ast}$,

\item 
$f\leq g$ iff $f\left( r,\infty\right) \leq g\left( r.\infty\right) $
for all $r\in\mathbb{R}$ iff $f\left( -\infty,r\right) 
\geq g\left(-\infty,r\right) $ for all $r\in\mathbb{R}$.
\end{enumerate}
\end{lemma}

Lemma \ref{Lem:1} gives necessary and sufficient conditions for a function
on right rays to be extended to a frame map.

\begin{lemma}[{\protect\cite[3.1.2]{BallHager:1991}}]
\label{Lem:1} A function $f:\left\{ \left(r,\infty\right) :r\in\mathbb{R}%
\right\} \rightarrow L$ can be extended to an element of $\mathcal{R}L$ iff
it satisfies the following conditions for all $r,s\in\mathbb{R}$. The
extension is unique when it exists.

\begin{enumerate}
\item $f\left( s,\infty\right) \prec f\left( r,\infty\right) $ whenever $r<s$%
.

\item $f\left( r,\infty\right) =\bigvee_{s>r}f\left( s,\infty\right) $.

\item $\bigvee_{r}f\left( r,\infty\right) =\bigvee_{r}f\left(
r,\infty\right) ^{\ast}=\top$.
\end{enumerate}
\end{lemma}

We provide a proof of Corollary \ref{Cor:1} in order to illustrate the use
of Theorem \ref{Thm:2} in calculations. This sort of reasoning will get
heavy use in what follows. For $f \in \mathcal{R}L$, the cozero element of $f
$ is 
\begin{equation*}
\coz f \equiv f(\mathbb{R}\smallsetminus \{0\}) = |f|(0,\infty).
\end{equation*}

\begin{corollary}\label{Cor:1}
For $f,g\in\mathcal{R}L$ and $c,r\in\mathbb{R}$,

\begin{enumerate}
\item $\left( f-c\right) \left( r,\infty\right) =f\left( c+r,\infty \right) $%
.

\item $\coz f^{+} =\left( f\vee0\right) \left( \mathbb{R}\smallsetminus\left%
\{ 0\right\} \right) =f\left( 0,\infty\right) $.

\item $\coz \left( f-c\right) ^{+} =\left( f-c\right) \left( 0,\infty\right)
=f\left( c,\infty\right) $.

\item $\left( f\wedge g\right) \left( r,\infty\right) =f\left(
r,\infty\right) \wedge g\left( r,\infty\right) $.

\item For $f,g\geq0$, $\bigvee_{\mathbb{N}}\coz \left( nf-g\right) ^{+} =%
\coz f$.
\end{enumerate}
\end{corollary}

\begin{proof}
To prove (1), consider $U\in\mathcal{O}\mathbb{R}$. Then by Theorem \ref%
{Thm:2} we have 
\begin{equation*}
\left( f-c\right) \left( r,\infty\right) = \hspace{-10pt}%
\bigvee_{U_{1}-U_{2}\subseteq\left( r,\infty\right) } \hspace{-10pt}\left(
f\left( U_{1}\right) \wedge c\left( U_{2}\right) \right) .
\end{equation*}
But if $U_{1}-U_{2}\subseteq\left( r,\infty\right) $ then $U_{1}$ is bounded
below and $U_{2}$ is bounced above, say $U_{1}\subseteq\left( t,\infty
\right) $ and $U_{2}\subseteq\left( -\infty,t-r\right) $ for some $t\in%
\mathbb{R}$. And if, in addition, and $c\left( U_{2}\right) >\bot$ then $%
c\in U_{2}$, hence $t>c+r$. That is to say that 
\begin{equation*}
\left( f-c\right) \left( r,\infty\right) = \bigvee_{t>c+r}
f\left(t,\infty\right) =f\left( c+r,\infty\right) .
\end{equation*}

The proof of (2) is similar to the proof of (1), and (3) follows from (1)
and (2). To prove (4), again consider $U\in O\mathbb{R}$. 
\begin{equation*}
\left( f\wedge g\right) \left( r,\infty\right) = \hspace{-10pt}%
\bigvee_{U_{1}\wedge U_{2}\subseteq\left( r,\infty\right) } \hspace{-10pt}%
\left( f\left( U_{1}\right) \wedge g\left( U_{2}\right) \right) .
\end{equation*}
But $U_{1}\wedge U_{2}\subseteq\left( r,\infty\right) $ iff $%
U_{1}\subseteq\left( r,\infty\right) $ and $U_{2}\subseteq\left( r,\infty
\right) $. Hence$\left( f\wedge g\right) \left( r,\infty\right) =f\left(
r,\infty\right) \wedge g\left( r,\infty\right) $.

To verify (5), note that 
\begin{equation*}
\coz \left( nf-g\right) ^{+} =\left( nf-g\right) \left(0,\infty\right) = 
\hspace{-10pt}\bigvee_{nU_{1}-U_{2}\subseteq\left( 0,\infty\right)} \hspace{%
-10pt}\left( f\left( U_{1}\right) \wedge g\left( U_{2}\right) \right) .
\end{equation*}
But if $nU_{1}-U_{2}\subseteq\left( 0,\infty\right) $ then $%
U_{1}\subseteq\left( r,\infty\right) $ and $U_{2}\subseteq\left( -\infty
,nr\right) $ for some $r\in\mathbb{R}$, so that 
\begin{align*}
\bigvee_{\mathbb{N}}\coz \left( nf-g\right) ^{+} &= \bigvee_{\mathbb{N}}
\bigvee_{\mathbb{R}} \left( f\left(r,\infty\right) \wedge g\left(
-\infty,nr\right) \right) = \bigvee_{\mathbb{R}} \bigvee_{\mathbb{N}} \left(
f\left( r,\infty\right)\wedge g\left(-\infty,nr\right) \right) \\
&=\bigvee_{\mathbb{R}}\left( f\left( r,\infty\right) \wedge \bigvee_{\mathbb{%
N}}g\left( -\infty,nr\right) \right) = \bigvee_{\mathbb{R}}f\left(
r,\infty\right) = \coz f. \qedhere
\end{align*}
\end{proof}

\subsection{The frame of $\mathbf{W}$-kernels of $A$\label{Subsec:2}}

Most of the calculation takes place in the frame of $\mathbf{W}$-kernels of $%
G$. The basic facts concerning this frame are well known; we briefly review
them here to fix notation.

\begin{lemma}\label{Lem:4}
Let $K$ be a convex $\ell$-subgroup of $G$.

\begin{enumerate}
\item $G/K$ is archimedean iff  
\begin{equation*}
\left( \forall~ n\in\mathbb{N~}\left( \left( nf-g\right) ^{+} \in K\right)
\Longrightarrow f\in K\right),  \qquad f,g \in G^+.
\end{equation*}

\item 
$K$ is a $\mathbf{W}$-kernel if, in addition, 
\[
g\wedge 1\in K \implies g\in K, \qquad g \in G^+.
\]
\end{enumerate}
\end{lemma}

\begin{proof}
(1) We have
\begin{equation*}
\left( nf-g\right) ^{+} \in K \iff K+\left( nf-g\right) ^{+}
=K\iff K+nf\vee b=K+g\iff K+nf\leq K+g.
\end{equation*}
This makes it clear that the condition displayed in (1) is equivalent to the
archimedean property of the quotient $G/K$.

(2) This is evidently a reformulation of the requirement that $K+u$ should
function as a weak unit of the quotient, i.e., that $\left( K+g\right)
\wedge \left( K+u\right) =0$ imply $K+g=0$.
\end{proof}

\begin{corollary}\label{Cor:4}
Suppose $G$ is bounded. Then a convex $\ell$-subgroup $K$ 
is a proper $\mathbf{W}$-kernel iff
\begin{enumerate}
\item
$
\forall~n \in \mathbb{N}~((nf - 1)^+ \in K) \implies f \in K,\ f \in G^+,$ 
and 
\item
$1 \notin K$.
\end{enumerate}
In particular, $[g] = \{h : \forall~n~\exists~m~(n|h| - 1)^+ \leq mg\}$,
$g \in G^+$. 
\end{corollary}

\begin{proof}
It is straightforward to show that in condition (1) of Lemma \ref{Lem:4}, 
the element $g$ may be chosen to be $1$ if $G$ is bounded.  What we must
also demonstrate is that condition (2) above implies condition (2) of Lemma
\ref{Lem:4}. Given $g \in G^+$, find a positive integer $n$ such that 
$g \leq n$. Then $g \wedge 1 \in K$ implies $ng \wedge n \in K$ 
because $K$ is a group, hence $g \wedge n \in K$ because $K$ is convex, 
with the result that $g \in K$.
\end{proof}

Since $\mathbf{W}$ is closed under products, 
the intersection of an
arbitrary family of $\mathbf{W}$-kernels is itself a $\mathbf{W}$-kernel. We
denote the $\mathbf{W}$-kernel generated by a subset $S\subseteq G$ by 
\begin{equation*}
\left[ S\right] \equiv\bigcap\left\{ K:K\text{ is a }\mathbf{W}\text{-kernel
and }S\subseteq K\right\} .
\end{equation*}

\begin{definition}
The frame of $\mathbf{W}$-kernels of $G$ is called the \emph{Madden frame of 
$G$}; we denote it by $\mathcal{M}G$.
\end{definition}

\begin{lemma}
$\mathcal{M}G$ forms a regular Lindel\"{o}f frame under the inclusion order.
Its operations are 
\begin{equation*}
K_{1}\wedge K_{2} =K_{1}\cap K_{2}\text{ \ and\ \ } 
\bigvee_{I}K_{i} 
= \left[K_i : i \in I \right]
= \left[\bigcup_{I}K_{i}\right] .
\end{equation*}
\end{lemma}

\begin{proof}
\cite[3.2.2, 3.25]{BallHager:1991}.
\end{proof}

\subsection{The Madden representation for $\mathbf{W}$}

\label{Subsec:1}

Let $G$ be a $\mathbf{W}$-object with $L \equiv \mathcal{M}G$ its frame of $%
\mathbf{W}$-kernels. For each $g \in G$ and $r \in \mathbb{R}$, define 
\begin{equation*}
\widehat{g}(r,\infty) \equiv \left[ \left( g - r \right)^+\right].
\end{equation*}
Thus defined, $\widehat{g}$ satisfies the requirements of Lemma \ref{Lem:1},
and thus extends to a unique frame map $\mathcal{O}\mathbb{R} \to L$, which
we also denote $\widehat{g}$. We write $\widehat{G}$ for $\{\widehat{g} : g
\in G\} $, and $\mu_G : G \to \widehat{G}$ for the mapping $g \mapsto 
\widehat{g}$.

We say that a $\mathbf{W}$-morphism $\theta : H \to \mathcal{R}M$ 
is \emph{cozero dense} if 
\begin{equation*}
a = \bigvee_{\substack{ h \in H  \\ \coz \theta(h) \leq a}} \coz \theta(h), 
                   \qquad a \in M.
\end{equation*}
Note that it is enough for this condition to hold for each $a \in \coz M$
because $M$ is assumed to be completely regular.  

\begin{theorem}[\protect\cite{MaddenVermeer:1986}]\label{Thm:6}
Let $G$, $L$, $\widehat{G}$, and $\mu _{G}$ have the meaning above.

\begin{enumerate}
\item Then $\mu_G$ is a cozero dense $\mathbf{W}$-injection, 
and its range restriction $G \to \widehat{G}$ is a 
$\mathbf{W}$-isomorphism.

\item $L$, $\widehat{G}$, and $\mu_G$ are unique up to isomorphism with
respect to their properties in (1).

\item For any frame $M$ and $\mathbf{W}$-morphism $\theta$ there is a unique
frame map $k$ making the diagram commute. 
\begin{figure}[h]
\setlength{\unitlength}{4pt}
\par
\begin{center}
\begin{picture}(36,12)(3,1)
\small
\put(0,12){\makebox(0,0){$G$}}
\put(12,12){\makebox(0,0){$\mathcal{R}L$}}
\put(12,0){\makebox(0,0){$\mathcal{R} M$}}
\put(24,12){\makebox(0,0){$L$}}
\put(24,0){\makebox(0,0){$M$}}
\put(36,12){\makebox(0,0){$\mathcal{O}\mathbb{R}$}}
\put(2,12){\vector(1,0){7.5}}
\put(2,10){\vector(1,-1){8}}
\put(12,10){\vector(0,-1){8}}
\put(24,10){\vector(0,-1){8}}
\put(33,12){\vector(-1,0){7}}
\put(34,10){\vector(-1,-1){8}}
\put(5,4.5){\makebox(0,0){$\theta$}}
\put(6,13.5){\makebox(0,0){$\mu_{G}$}}
\put(14.5,6){\makebox(0,0){$\mathcal{R} k$}}
\put(22.5,6){\makebox(0,0){$k$}}
\put(30,13.75){\makebox(0,0){$\widehat{g}$}}
\put(32.5,4.5){\makebox(0,0){$\theta(g)$}}
\end{picture}
\end{center}
\end{figure}

\item
$k$ is surjective iff $\theta$ is cozero dense, and $k$ is one-one iff, for all 
$K \subseteq G^+$, $\bigvee_K \coz \theta(g) = \top $ in $M$ implies 
$\bigvee_K \coz \widehat{g} = \top$ in $L$. 
\end{enumerate}
\end{theorem}

\begin{proof}
A detailed proof may be found in \cite{BallHager:1991}; we comment only 
on part (4). For any $\mathbf{W}$-kernel $K \subseteq G$, 
\begin{align*}
k(K) &= k\left(\bigvee_{K^+}[g]\right) 
= k\left(\bigvee_{K^+}\coz\widehat{g}\right)
= k\left(\bigvee_{K^+}\widehat{g}(0,\infty)\right)
= \bigvee_{K^+}k\circ \widehat{g}(0,\infty)
= \bigvee_{K^+}\theta(g)(0,\infty) \\
&= \bigvee_{K^+}\coz\theta(g).
\end{align*}
This makes the surjectivity condition clear; the injectivity condition follows from
the fact that a frame morphism between regular frames is one-one
iff it is codense, i.e., iff the only element taken to the top of the codomain 
is the top element of the domain.  
\end{proof}

\section{Pointwise suprema defined}\label{Sec:4}

In dealing with continuous real-valued functions on a Tychonoff space $X$,
it is often important to know whether a given function $f$ is the supremum
of a given subset $K\subseteq \mathcal{C}X$, and, if so, whether this
supremum is pointwise, i.e., whether $\bigvee_{K}k\left( x\right) =f\left(
x\right) $ for all $x\in X$. In terms of the frame $\mathcal{O}X $ of open
sets of $X$, $f$ is the pointwise supremum of $K$ iff $\bigcup_{K}k^{-1}%
\left( r,\infty\right) =f\left( r,\infty\right) $ for all $r\in\mathbb{R}$.
It is the latter formulation which generalizes directly to the pointfree
setting.

\subsection{Pointwise suprema in $\mathcal{R}L$\label{Subsec:5}}

\begin{definition}
Let $L$ be a frame, and let $K$ be a subset and $f$ an element of $\mathcal{R%
}L$. We say that $f$ is the \emph{pointwise supremum (infimum) of $K$}, and
write $f = \bigvee^\bullet K$ ($f = \bigwedge^\bullet K$), provided that $%
f\left(r,\infty\right) = \bigvee_{K} k \left(r,\infty\right)$ ($%
f\left(-\infty, r\right) = \bigvee_{K} k \left(-\infty,r\right)$ ) holds in $%
L$ for all $r\in\mathbb{R}$.
\end{definition}

\begin{remarks}
\label{Rem:1}A few remarks about this definition are in order.

\begin{enumerate}
\item Observe that the frame definition coincides with the spatial
definition in case $L$ is the topology of a Tychonoff space.

\item Recall that by Lemma \ref{Lem:3} a frame map is completely determined
by its values on the right or left rays alone. This makes the appearance of
only the rays in this definition less mysterious.

\item Recall that by Lemma \ref{Lem:3} an element $g\in\mathcal{R}L$ lies
above (below) each $k\in K$ iff $k\left( r,\infty\right) \leq
g\left(r,\infty\right) $ ($g\left( -\infty,r\right) \leq k\left( -\infty
,r\right) $) for all $r\in\mathbb{R}$ and all $k \in K$.

\item It follows from the preceding remarks that $f=\bigvee K$ whenever $f =
\bigvee^\bullet K$, and dually.

\item It follows from the preceding remarks that if $f=\bigvee^{\bullet }K=g$
then $f=g$.
\end{enumerate}
\end{remarks}

We list some of the nice properties of pointwise suprema and infima .

\begin{proposition}
\label{Prop:8}Let $F$ and $K$ be subsets and let $f_{0}$ and $k_{0}$ be
elements of $\mathcal{R}L$.

\begin{enumerate}
\item $f_{0} = \bigvee^\bullet F $ iff $-f_{0} = \bigwedge^\bullet \left(
-F\right) \equiv \bigwedge_{F}^\bullet \left( -f \right)$, and dually.

\item $f_{0}=\bigvee^\bullet \left\{ f_{0}\right\} = \bigwedge^\bullet
\left\{f_0\right\}$.

\item If $f_{0} = \bigvee^\bullet F$ and $k_{0} = \bigvee^\bullet K$ then $%
f_{0}\boxdot k_{0} = \bigvee_{F,K}^\bullet \left( f\boxdot k\right) $, where 
$\boxdot$ stands for one of the $\ell$-group operations $+$, $\vee$, or $%
\wedge$.

\item If $f_{0} = \bigvee^\bullet F$ and $0\leq r\in\mathbb{R}$ then $rf_{0}
= \bigvee_{F}^\bullet rf$.
\end{enumerate}
\end{proposition}

\begin{proof}
(1) follows from the fact that $\left( -f\right) \left(
-\infty,r\right)=f\left( -r,\infty\right) $ for any $f\in \mathcal{R}L$ and $%
r\in\mathbb{R}$, as can be readily checked with the aid of Theorem \ref%
{Thm:2}. (2) is trivial. To prove part (3) for the $+$ operation, first
observe that for $r\in\mathbb{R}$, 
\begin{gather*}
\bigvee_{F,K}\left( f+k\right) \left( r,\infty\right) =
\bigvee_{F,K}\bigvee_{ U_{1}+U_{2}\subseteq\left( r,\infty\right)}
\left(f\left( U_{1}\right) \wedge k\left( U_{2}\right) \right) .
\end{gather*}
But if $U_{1}+U_{2}\subseteq\left( r,\infty\right) $ then both $U_{i}$'s are
bounded below, say $U_{1}\subseteq\left( s,\infty\right) $ and $%
U_{2}\subseteq\left( r-s,\infty\right) $ for some $s\in\mathbb{R}$.
Therefore this join works out to 
\begin{align*}
\bigvee_{F,K}\bigvee_s\left(f(s,\infty) \wedge k(r - s, \infty)\right) &=
\bigvee_s\bigvee_{F,K}\left(f(s,\infty) \wedge k(r - s, \infty)\right) =
\bigvee_s \left(f_0(s, \infty) \wedge k_0(r - s, \infty)\right) \\
&= \left(f_0 + k_0 \right)(r, \infty).
\end{align*}
The proofs of (3) for the join and meet operations are similar. Finally, to
verify (4) simply note that if $r>0$ then $\left( rf\right)
\left(s,\infty\right) = f\left( s/r,\infty\right) $ for $f\in \mathcal{R}L$
and $s\in\mathbb{R}$, as may be easily seen using Theorem \ref{Thm:2}.
\end{proof}

\subsection{Pointwise suprema in $\mathbf{W}$\label{Subsec:6}}

Having formulated the notion of pointwise supremum in $\mathcal{R}L$, let us
now generalize it to abstract $\mathbf{W}$-objects.

\begin{definition}[First definition of pointwise supremum in $\mathbf{W}$]
For $F\subseteq G\in\mathbf{W}$ and $f_{0}\in G$, we shall say that \emph{$%
f_{0}$ is the pointwise supremum (infimum) of $F$}, and write $f_{0}
=\bigvee^\bullet F$ ($f_0 = \bigwedge^\bullet F$), if the corresponding
statement holds in $\widehat{G}$, i.e., if $\widehat{f_{0}} =
\bigvee_{F}^\bullet \widehat{f}$ ($\widehat{f_{0}} = \bigwedge_{F}^\bullet 
\widehat{f}$).
\end{definition}

Pointwise suprema can be characterized concretely
by use of the details of the Madden representation (see Subsection 
\ref{Subsec:1}).

\begin{proposition}\label{Prop:1}
Let $F$ be a subset and $f_{0}$ an element of a $\mathbf{W}$-object $G$.
Then 
\begin{align*}
f_{0} & = {\bigvee\nolimits}^\bullet F \iff \forall~r \in \mathbb{R}~\left(%
\left[ \left(f-r\right)^{+} :f\in F\right] = \left[ \left( f_{0}-r\right)
^{+} \right] \right), \\
f_{0} & = {\bigwedge\nolimits}^\bullet F \iff \forall~r \in \mathbb{R}~\left(%
\left[ \left(r-f\right)^{+} :f\in F\right] = \left[ \left( r-f_{0}\right)
^{+} \right]\right).
\end{align*}
\end{proposition}

\begin{proof}
In the Madden representation $G\rightarrow\widehat{G}$, 
\[
\widehat{f}\left(r,\infty\right) =\left[ \left( f-r\right) ^{+} \right]
\quad \text{and} \quad 
\widehat{f}\left( -\infty,r\right) =\left[ \left( r-f\right)^{+} \right]. \qedhere
\]
\end{proof}

$\mathbf{W}$-morphisms preserve pointwise suprema.

\begin{proposition}\label{Prop:7}
If $\theta:G\rightarrow H$ is a $\mathbf{W}$-morphism and if $%
f_{0} = \bigvee^\bullet F$ in $G$ then $\theta\left( f_{0}\right) = \bigvee
_{F}^\bullet \theta\left( f\right) $ in $H$.
\end{proposition}

\begin{proof}
Identify $G$ and $H$ with their Madden representations in $\mathcal{R}L$ and 
$\mathcal{R}M$, where $L \equiv \mathcal{M}G$ and $M \equiv \mathcal{M} H$.
Then there is a unique frame map $\mathcal{M}\theta\equiv k:L\rightarrow M$
which realizes $\theta$ in the sense that $\theta\left(g\right) \left(
U\right) =k \circ g\left( U\right) $ for all $U\in\mathcal{O}\mathbb{R}$ and 
$g\in G$. Therefore we have, for $r\in\mathbb{R}$,%
\begin{align*}
\bigvee_{F}\theta\left( f\right) \left( r,\infty\right) &= \bigvee_{F}k
\circ f\left( r,\infty\right) = k \left( \bigvee_{F}f\left(
r,\infty\right)\right) = k \circ f_{0}\left(r,\infty\right) = \theta\left(
f_{0}\right) \left(r,\infty\right) . \qedhere
\end{align*}
\end{proof}

It is a surprising fact that the converse of Proposition \ref{Prop:7} holds
as well. In general, the supremum of a subset $F$ of a $\mathbf{W}$-object $%
G $ depends on the context. If, for instance, $G$ is a subobject of $H$, it
may well happen that $f_{0}=\bigvee F$ for some $f_{0}\in G$ but $%
f_{0}\neq\bigvee F$ in $H$. The point of Proposition \ref{Prop:9} is that it
is precisely the pointwise suprema which are context free.

\begin{proposition}
\label{Prop:9}Let $F$ be a subset and $f_{0}$ an element in some 
$\mathbf{W}$-object $G$. Then $f_{0} = \bigvee^\bullet F$ iff $\theta\left( f_{0}\right)
= \bigvee_{F}\theta\left( f\right) $ for every $\mathbf{W}$-morphism $\theta$
out of $G$, and dually.
\end{proposition}

\begin{proof}
Proposition \ref{Prop:7} is the forward implication of this equivalence. So
suppose $f_{0}$ is not the pointwise supremum of $F$, let $L$ be the Madden
frame of $G$, and identify $G$ with its Madden representation 
$\widehat{G}\leq\mathcal{R}L$. We may assume that $f_{0}=0$, 
since otherwise we may
replace $F$ by $F-f_{0}\equiv\left\{ f-f_{0}:f\in F\right\} $ by 
Proposition~\ref{Prop:8}. We must find a $\mathbf{W}$-morphism 
$\theta:G\rightarrow H$ such that $\bigvee _{F}\theta\left( f\right) \neq0$.

Let $i:L\rightarrow M$ be a frame embedding of $L$ into a boolean frame $M$.
(Such an embedding exists; see \cite[II, 2.6]{Johnstone:1982}.) Since 
$\bigvee^\bullet  F \neq0$ there exists some 
$r\in\mathbb{R}$ such that 
\begin{equation*}
a\equiv{\bigvee\nolimits}_{F}f\left( r,\infty\right) <0\left(
r,\infty\right) =\left\{ 
\begin{array}{ll}
\bot & \text{if }r\geq0 \\ 
\top & \text{if }r<0%
\end{array}
\right. .
\end{equation*}
Note that $0\left( r,\infty\right) $ must be $\top$, hence $r<0$. 
Now $i\left( a\right) $ has complement $b$ in $M$; note that 
$b>\bot$ because $i\left( a\right) <\top$ since $a<\top$ and $i$ 
is one-one.

Let $k:M\rightarrow{}\downarrow\!b$ designate the open quotient frame 
map $c\longmapsto c\wedge b$, $c\in M$, let 
$H\equiv\mathcal{R}\left({\downarrow b}\right) $, and let 
$\psi\equiv\mathcal{R}(k\circ i):\mathcal{R}L\to H$. 
We claim that the desired map $\theta$ is the restriction of $\psi$ to 
$\widehat{G} \approx G$. For if $f\in F$ then
\begin{equation*}
\psi\left( f\right) \left( r,\infty\right) = k \circ i \circ f \left(
r,\infty\right) = i \circ f\left( r,\infty\right) \wedge b=\bot
\end{equation*}
since $i \circ f \left( r,\infty\right) \leq i\left( a\right) $ and 
$i\left(a\right) \wedge b=\bot$. It follows that for $s\in\mathbb{R}$, 
\begin{equation*}
\psi\left( f\right) \left( s,\infty\right) 
\leq\left( r/2\right) \left(s,\infty\right) 
=\begin{cases}
     \bot & \text{if $s\geq r/2$} \\ 
     \top & \text{if $s<r/2$} \end{cases},
\end{equation*}
which implies by Lemma \ref{Lem:3}(2) that $\psi\left( f\right) \leq r/2<0$
for all $f\in F$, meaning that $\bigvee_{F}\psi\left( f\right) \neq0$. This
completes the proof.
\end{proof}

Some caution is required when dealing with pointwise suprema. If $F\subseteq
G$ and $f_{0}\in G$ are such that $f_{0}=\bigvee^{\bullet }F$ in some $%
\mathbf{W}$-extension $H\geq G$ then $f_{0}=\bigvee F$ in $G$, of course,
but the join may not be pointwise in $G$.

\begin{example}
Let $X$ be $\omega+1$, the one-point compactification of the discrete space
of finite ordinals. Let $G$ be $\mathcal{C}X$ and let 
\begin{equation*}
H\equiv\left\{ g+rh_{0}:g\in G,~r\in\mathbb{R}\right\} ,
\end{equation*}
where $h_{0}\equiv\left( n\longmapsto n\right) $ and $h_{0}\left(\omega%
\right) =\infty$. $H$ is a $\mathbf{W}$-object in $DX$. Let $F$ be the
family of functions 
\begin{equation*}
f_{n}\left( k\right) \equiv 
\begin{cases}
1 & \text{if $k\leq n$} \\ 
0 & \text{if $k>n$}%
\end{cases}
\qquad n<\omega.
\end{equation*}
Then it is not hard to check that $\bigvee^\bullet F=1$ in $H$ and $\bigvee
F=1$ in $G$ but the latter join is not pointwise.
\end{example}

For emphasis, we recast the definition of pointwise supremum in an arbitrary 
$\mathbf{W}$-object.

\begin{definition}[Second definition of pointwise supremum in $\mathbf{W}$]
For $F\subseteq G\in\mathbf{W}$ and $f_{0}\in G$, we shall say that 
\emph{$f_{0}$ is the pointwise supremum (infimum) of $F$}, and write 
$f_{0}=\bigvee^\bullet F$ ($f_0 = \bigwedge^\bullet F$), if 
$\bigvee\theta[F]) = \theta(f_0)$ ($\bigwedge \theta(F) = \theta(f_0)$) 
for all $\mathbf{W}$-homomorphisms $\theta : G \to H$.
\end{definition}

\section{Pointwise suprema applied}\label{Sec:8}

In this section we aim to show that pointwise suprema are useful for 
characterizing important attributes of a $\mathbf{W}$-object and its 
Madden frame. We begin by using them to characterize 
those $\mathbf{W}$-objects
in which every (countable) supremum is pointwise.  
Throughout this section $G$ will represent a $\mathbf{W}$-object 
with Madden frame $L$. 

\subsection{When all existing (countable) suprema are pointwise} 
Pointwise suprema are useful for detecting whether the Madden frame of a given 
$\mathbf{W}$-object is boolean or a $P$-frame. We shall require this information 
in Section \ref{Sec:6}. 

\begin{theorem}\label{Thm:7}
Suppose that $\bigwedge \mathbb{R}^+(G) = 0$.
Then all existing (countable) suprema in $G$ are pointwise
iff $L$ is boolean (a $P$-frame).
\end{theorem}

\begin{proof}
We prove this theorem in the boolean case; the same proof,
mutatis mutandis, works in the $P$-frame case. 
Assume $L$ is boolean, suppose $f$ is an element and $K$ is a subset of 
$G$ such that $f = \bigvee K$ in $G$, and assume for the sake of argument that 
$f(r, \infty) >\bigvee_K k(r, \infty) \equiv b$ for some real number $r$. 
Since $\bigvee_{s> r}f(s, \infty) = f(r, \infty)$, there is some $s > r$ 
for which $f(s,\infty) \nleq b$. Because $L$ is boolean 
$a \equiv f(s, \infty) \wedge b^* > \bot$;  
define the \enquote*{characteristic function} 
\begin{equation*}
\chi (t,\infty) \equiv 
\begin{cases}
\top \text{ if } t < 0, \\ 
a \text{ if } 0 \leq t <s - r, \\ 
\bot \text{ if } t \geq 1 ,%
\end{cases}
\qquad t \in \mathbb{R}, 
\end{equation*}
and check that $\chi$ satisfies the hypotheses of Lemma \ref{Lem:1} and so
extends to a unique member of $\mathcal{R}L$, and that, moreover, $\chi > 0$.

We claim that $f - k \geq \chi$ for all $k \in K$. To verify the claim first note 
that by Theorem \ref{Thm:2} 
$(f - k)(t,\infty) = \bigvee_{U_i}(f(U_1) \wedge k(U_2))$, 
where the join ranges over open subsets 
$U_i \subseteq \mathbb{R}$ such that $U_1 - U_2 \subseteq (0,\infty)$.
This condition implies that $U_1$ is bounded below and $U_2$ is bounded above, 
say $U_1 \subseteq (u,\infty)$ and $U_2 \subseteq (-\infty,u-t)$. Therefore
\begin{equation*}
(f - k)(t, \infty) =  \bigvee_u (f(u,\infty) \wedge k(-\infty, u - t)), 
\end{equation*}
If $t < 0$ then, since $f \geq k$ implies $f(-\infty, u-t) \leq k(-\infty, u - t)$, 
we get for any choice of $u \in \mathbb{R}$ that 
\begin{equation*}
(f - k)(t,\infty) \geq f(u,\infty) \wedge k(-\infty, u - t) \geq f(u,\infty)
\wedge f(-\infty, u - t) = \top. 
\end{equation*}
If $0 \leq t < s - r$ then $s - t > r$, hence $k(-\infty, s - t) \geq k(r,
\infty)^*$ since 
\begin{equation*}
k(-\infty, s - t) \vee k(r, \infty) = k\left((-\infty, s - t) \cup(r,
\infty) \right) = \top. 
\end{equation*}
This is relevant because $b^* = \left( \bigvee_K k(r,\infty)\right)^* =
\bigwedge_K k(r,\infty)^*$ as a result of the fact that $L$ is a complete
boolean algebra, so that. 
\begin{align*}
(f - k)(t,\infty) &\geq f(s,\infty) \wedge k(-\infty, s - t) \geq f(s,
\infty) \wedge k(r, \infty)^* \geq f(s, \infty) \wedge \bigwedge_L k(r,
\infty)^* \\
&= f(s, \infty) \wedge b^* = a = \chi(t,\infty).
\end{align*}
This proves the claim, which implies that $f > f - \chi \geq K$.  Since
$\mu_G : G \to \mathcal{R}L$ is cozero dense, there eists $0 < g \in G$
such that $\coz g \leq a$, and, by meeting $g$ with the appropriate constant
function $r \in \mathbb{R}^+(G)$ we may assume that $g \leq \chi$. In sum,
we have $f > f-g \geq K$, a violation of the assumption that 
$f = \bigvee K$ in $G$. Our only recourse is to conclude that 
$f(r,\infty) = \bigvee_K k(r,\infty)$ for all $r \in \mathbb{R}$, i.e.,
$f = \bigvee^\bullet K$.

Now suppose that all 
existing suprema in $G$ are pointwise; 
we aim to show that an arbitrary element $a \in L \equiv \mathcal{R}L$ 
is complemented. For that purpose define subsets 
\[
U \equiv \{g \in G : \coz g \leq a \text{ and $0\leq g \leq 1$}\}, 
\quad\text{and}\quad
V \equiv \{g \in G : \coz g \leq a^* \text{ and $0 \leq g \leq 1$}\}.
\]
By suitably augmenting $U$ we may assume that $0 \leq k \leq g \in U$
implies $k \in U$, and that $g \in U$ implies $ng \wedge 1 \in U$ 
for all $n$, and similarly for $V$. 

We claim that $\bigvee (U\cup V) = 1$.  If not then there exists some 
$k \in G$ such that $g \leq k < 1$ for all $g \in U \cup V$.
This means that $1 - k > 0$, hence $b \equiv \coz(1 - k) > \bot$. Since 
$\bigvee_{U \cup V}\coz g$ is a dense element of $L$, $b$ 
meets either $\bigvee_U \coz g$ or $\bigvee_V \coz g$ nontrivially, 
say $0 < g \in U$ is such that $\coz g \leq b$.  Since 
$\coz g = g(0,\infty) = \bigvee_m g(\frac{1}{m},\infty)$, 
there exists some positive integer $m$ for which 
$\bot < g(\frac{1}{m},\infty) = \coz (g - \frac{1}{m})^+ 
= \coz (mg -1)^+$.  But then
\begin{align*}
\coz(mg \wedge 1 -k)^+ &= \coz((mg - k)^+\wedge (1 - k))
=\coz(mg - k)^+ \wedge \coz(1 - k) \\
&= \coz(mg -k)^+ \wedge b
\geq \coz (mg - 1)^+ \wedge b = \coz (mg - 1)^+ > \bot.
\end{align*}
This is a contradiction, since $g \in U$ implies $mg \wedge 1 \in U$, hence
$mg \wedge 1 \leq k$. A similar argument covers the case in which
there exists some $0 < g \in V$ such that $\coz g \leq a^*$, and the two
cases together prove the claim.

Let $\bigvee_U \coz h \equiv u$ and $\bigvee_V \coz h \equiv v$.
Then because $1 = \bigvee^\bullet(U \cup V)$ we have 
\[
\top = 1(0,\infty) = \bigvee_{U \cup V}g(0,\infty)
= \bigvee_{U \cup V}\coz g = \bigvee_U\coz g \cup \bigvee_V\coz g 
= u \vee v. 
\]
Since $u \leq a$ and $v \leq a^*$ by construction, we see that $u$ and
$v$ are complementary, and that $u = a$.
\end{proof}

\subsection{Truncate sequences}

\label{Subsec:15}

The following fact plays an important role in our analysis of unconditional
pointwise completeness in Section \ref{Sec:6}.

\begin{proposition}\label{Prop:3}
For any $f\in G$, $\bigvee _{\mathbb{N}}^\bullet \left( f\wedge n\right) = f$.
\end{proposition}

\begin{proof}
Identify $G$ with $\widehat{G} \leq \mathcal{R}L$.
For any $r\in\mathbb{R}$ we have by Corollary \ref{Cor:1}(4) that 
\begin{equation*}
\left( f\wedge n\right) \left( r,\infty\right) =f\left( r,\infty\right)
\wedge n\left( r,\infty\right) =\left\{ 
\begin{array}{ll}
\bot & \text{if }r\geq n \\ 
f\left( r,\infty\right) & \text{if }r<n%
\end{array}
\right. .
\end{equation*}
Hence $\bigvee_{\mathbb{N}}\left( f\wedge n\right) \left( r,\infty\right) =
f\left( r,\infty\right) $ for all $r \in \mathbb{R}$.
\end{proof}

Proposition \ref{Prop:3} raises an important question: which sequences in 
$G$ are sequences of truncates of a member of $\mathcal{R}L$, so
called truncate sequences?

\begin{proposition}\label{Prop:12} 
Let $\{g_n\} \subseteq G^+$ be the sequence of truncates of 
$h \in\mathcal{R}L^+$, i.e., $\widehat{g}_n = h \wedge n $ for all $n$. Then

\begin{enumerate}
\item $g_{n + 1} \wedge n = g_n$ in $G$, and

\item $\bigvee_n \widehat{g}_n(-\infty,n) = \top$ in $L$.
\end{enumerate}

Conversely, any sequence in $G$ having these two properties is
the sequence of truncates of some $h \in \mathcal{R}L$.
\end{proposition}

\begin{proof}
If $g_n = h \wedge n$ for all $n$ then (1) obviously holds, and 
\begin{align*}
\bigvee_n g_n(-\infty,n) 
&= \bigvee_n (h \wedge n)(-\infty, n) 
= \bigvee_n h(-\infty,n) 
= h\left( \bigvee_n (-\infty,n) \right) \\
&= h(-\infty, \infty) = \top.
\end{align*}
Suppose now that $\{ g_n \}$ is a sequence in $\mathcal{R}L^+$ satisfying
(1) and (2). Put $h(-\infty, r) \equiv g_n(-\infty,r)$ for any $n >r$.
This definition is independent of the choice of $n$ by (1). We must show
that $h$ satisfies the properties in the (up-down dual of) 
Lemma \ref{Lem:1}. It is clear that $h$ satisfies the first of these properties,
namely that $h(-\infty, s) \prec h(-\infty,r)$ whenever $r < s$, because
it reduces to the same property of $g_n$ for sufficiently large $n$, and 
$h$ satisfies the second property for similar reasons. Since 
$\bigvee_r h(-\infty,r) = \bigvee_n h(-\infty,n) = \bigvee_n g_n(-\infty,n)$, 
$h$ also satisfies half of the third property. But $h$ also satisfies the
other half because 
$\bigvee_r h(-\infty,r)^* = \bigvee_r g_1(-\infty,r)^*= \top$.
\end{proof}

\begin{definition}[Truncate sequence]
We shall refer to a sequence $\{g_n\} \subseteq G$ satisfying
Proposition \ref{Prop:12} as a \emph{truncate sequence}.
\end{definition}

\begin{corollary}
\label{Cor:8} 
Every truncate sequence in $G$ has a pointwise join in $\mathcal{R}L$.
\end{corollary}

In Section 10 of \cite{Hager:2013}, the second author conducted an analysis
of a construct which is closely related to truncate sequences, but stronger.
His \enquote*{expanding sequences} have the first property of truncate
sequences but satisfy $\bigcap_n (u_{n + 1} - u_n)^{\bot\bot} = 0$ instead
of the second property. The possession of a supremum for every such sequence
turns out to be equivalent to the property of being *-maximum, or *-max for
short. A $\mathbf{W}$-object is *-max if it contains a copy of every other 
$\mathbf{W}$-object with the same bounded part. This interesting attribute is
not the same as requiring the truncate sequences to have joins, for it
implies, inter alia, that the classical Yosida space of $G$ be a quasi-$F$
space. As is evident from Corollary \ref{Cor:8}, no such restriction applies
to the $\mathbf{W}$-objects in which the truncate sequences have joins.

We confess ignorance of the many questions that arise naturally here,
postponing an investigation for the time being. But surely the first
question is unavoidable, as it is motivated by the characterization of the
divisible *-max $\mathbf{W}$-objects as being precisely those in which every
expanding sequence has a join (\cite[Section 10]{Hager:2013}).  
See Theorem \ref{Thm:9} below for further discussion of these topics.

\begin{question}
Which $\mathbf{W}$-objects $G$ have the feature that every truncate 
sequence in $G$ has a pointwise join in $G$?
\end{question}

\subsection{Pointwise closure and $\mathbf{W}$-kernels}

$\mathbf{W}$-kernels are characterized by the property of 
being closed under pointwise joins. A convex $\ell$-subgroup $K \leq G$ 
is said to by \emph{pointwise closed} if $K_0 \subseteq K^+$ and 
$\bigvee^\bullet K_0 = g$ imply $g \in K$.  

\begin{proposition}
A convex $\ell$-subgroup $K$ of a $\mathbf{W}$-object $G$ is a 
$\mathbf{W}$-kernel iff it is pointwise closed. 
\end{proposition}

\begin{proof}
Suppose $K$ is a $\mathbf{W}$-kernel with subset $K_0$ such that 
$\bigvee^\bullet K_0 = g$.  According to Proposition \ref{Prop:1} we 
are supposing that $[(k - r)^+ : k \in K_0] 
= [(g - r)^+]$ for all $r \in \mathbb{R}$.  
In particular, for $r = 0$ this says that $[K_0] = [g]$, i.e., 
any $\mathbf{W}$-kernel containing $K_0$ must also contain $g$.
But one such $\mathbf{W}$-kernel is $K$, hence $g \in K$ and $K$
is pointwise closed.

Now suppose that $K$ is a pointwise closed convex $\ell$-subgroup of $G$;
we must show that $K$ has properties (1) and (2) of Lemma \ref{Lem:4}.
To check (2), suppose that $g \wedge 1 \in K$ for some $g \in G$. Then
for each positive integer $n$ we would have $ng \wedge n \in K$ 
because $K$ is a group, hence $g \wedge n \in K$ because $K$ is convex,
with the result that $g \in K$ by Proposition \ref{Prop:3}.  

To check property
(1) consider $a,b \in G^+$ such that $(na - b)^+ \in K$ for all $n$. 
We claim that $\bigvee^\bullet_n ((na - b)^+ \wedge a) = a$.
What we will actually prove is that 
\[\bigvee\nolimits^\bullet_n (((n - 1)a - b) \vee (-a)) \wedge 0) = 0,
\]
the result of subtracting $a$ from the equation claimed.  Since 
$0(r,\infty) = \top$ for $r < 0$ and $\bot$ otherwise, this amounts to 
showing that $\bigvee_n((n - 1)a - b) \vee (-a)) (r,\infty) = \top$ 
for $r < 0$. According to Theorem \ref{Thm:2}, it is sufficient to 
demonstrate that $\bigvee_n\bigvee_{U,V}(a(U) \wedge b(V)) = \top$, 
where the inner join ranges over open
subsets $U,V \subseteq \mathbb{R}$ for which 
\[
((n - 1)U - V) \vee (-U) \subseteq (r,\infty).
\]
But if $U$ and $V$ are to satisfy this containment then $U$ must be 
bounded both above and below, say $U \subseteq (u,w)$, and $V$
must be bounded above, say $V \subseteq (-\infty,v)$, where 
$(n-1)u > r + v$ or $-w > r$.  In sum, we must show that, for $r < 0$, 
\begin{align*}
\top &=
\bigvee_n\left(\bigvee_{u < w < -r}\bigvee_v(a(u,w) \wedge b(-\infty,v)) \vee 
 \bigvee_v\bigvee_{w>u>\frac{r + v}{n - 1}}
     (a(u,w) \wedge b(-\infty,v))\right) \\
&= \bigvee_n\left(\bigvee_{u < w < -r}\left(a(u,w) 
     \wedge \bigvee_v b(-\infty,v)\right) \vee 
     \bigvee_v \left(\bigvee_{w>u>\frac{r + v}{n - 1}} a(u,w) 
     \wedge b(-\infty,v)\right)\right) \\
&= \bigvee_n\left(\bigvee_{u < w < -r}a(u,w)  \vee 
      \bigvee_v \left(a(\frac{r + v}{n - 1},\infty) 
     \wedge b(-\infty,v)\right)\right) \\
&= a(-\infty,-r) \vee\bigvee_v\bigvee_n\left(a(\frac{r + v}{n - 1},\infty) 
     \wedge b(-\infty,v)\right)  
\end{align*}
But for $v > -r$ we have 
$\bigvee_n\left(a(\frac{r + v}{n - 1},\infty) 
     \wedge b(-\infty,v)\right) = a(0,\infty) \wedge b(-\infty,v)$, so that 
the last join displayed reduces to 
$a(-\infty,-r) \vee a(0,\infty) = a(-\infty,\infty) = \top$, thereby proving the 
claim and the proposition.
\end{proof}

\section{Conditional pointwise completeness\label{Sec:5}}

The classical Nakano-Stone Theorem asserts that every bounded (countable)
subset of $\mathcal{C}X$ has a supremum in $\mathcal{C}X$ iff $X$ is
extremally disconnected (basically disconnected). In this section we prove
the corresponding result for pointwise suprema, Theorem \ref{Thm:3}. Our
analysis will be closely intertwined with the pointfree version of the
classical theorem, a result of Banaschewski and Hong \cite%
{BanaschewskiHong:2003}.

\subsection{The Banaschewski-Hong Theorem}

We begin with an observation.

\begin{proposition}\label{Prop:16}
A conditionally pointwise complete ($\sigma$-complete) $%
\mathbf{W}$-object is conditionally complete ($\sigma$-complete).
\end{proposition}

\begin{proof}
This follows from Remark \ref{Rem:1}(4).
\end{proof}

The converse of Proposition \ref{Prop:16} does not hold, even for 
$\mathbf{W}$ objects of the form $\mathcal{C} X $, $X$ a Tychonoff space. 
In this case $\mathcal{C} X $ is conditionally $\sigma$-complete iff 
$X$ is basically disconnected; this is the classical Nakano-Stone Theorem. 
On the other hand, if $X$ is compact and basically disconnected then
$G \equiv \mathcal{C}X$ is conditionally $\sigma$-complete. But  
$L \equiv \mathcal{M}G = \mathcal{O}X$ is a $P$-frame iff $X$ 
is a $P$ space, and a compact $P$-space is finite.
The point is that, by Theorem \ref{Thm:3}, $G$ is not conditionally 
pointwise $\sigma$-complete unless $X$ is finite. See also  
\cite[4N]{GillmanJerison:1960}.

Theorem \ref{Thm:1} is a modestly embellished version of the pointfree
Nakano-Stone Theorem, a beautiful result of Banaschewski and Hong \cite%
{BanaschewskiHong:2003}. We prove Theorem \ref{Thm:1} in some detail, not
just because we need the result but also because the proofs provide the
basis for the corresponding result for pointwise completeness in Subsection 
\ref{Subsec:4}.

Let us review the basic definitions: a frame is said to be extremally
disconnected (basically disconnected) provided that $a^{\ast}\vee
a^{\ast\ast}=\top$ for all $a\in L$ ($a\in\coz L$). And 
$\mathbb{R}^{+} \left( G\right) $ stands for the set of positive real 
numbers such that the corresponding constant function lies in $G$.

\begin{theorem}
\label{Thm:1}Let $G$ be a $\mathbf{W}$-object with Madden frame $L$. Then
conditions (1) and (2) together are equivalent to conditions (3) and (4).

\begin{enumerate}
\item 
$\bigwedge \mathbb{R}^{+} \left( G\right) = 0$, i.e., $G$ contains 
arbitrarily small positive multiples of $1$.

\item 
$G^{\ast }$ is conditionally complete ($\sigma $-complete).

\item 
$L$ is extremally disconnected (basically disconnected).

\item 
The Madden representation carries $G^*$ onto $\mathcal{R}^*L$.
\end{enumerate}
\end{theorem}

\begin{proof}
Since it is most relevant to our purposes, we prove the version of this
theorem having to do with the conditional $\sigma $-completeness of $G^{\ast
}$ versus the basic disconnectivity of $L$. The implication from (3) and (4)
to (1) and (2) is provided by the result of Banaschewski and Hong 
(\cite[Prop. 2]{BanaschewskiHong:2003}), since they prove that if 
$L $ is basically disconnected then $\mathcal{R}L$, and hence 
$\mathcal{R}^{\ast }L$, is conditionally $\sigma $-complete. 
The opposite implication is
provided by Propositions \ref{Prop:4} and \ref{Prop:2}. The running
assumptions throughout are that $L$ is the Madden frame of $G$, and that $G$
has been identified with its Madden representation in $\mathcal{R}L$, i.e., 
$G$ is a $\mathbf{W}$ subobject of $\mathcal{R}L$.
\end{proof}

\begin{proposition}\label{Prop:4}
If $\bigwedge \mathbb{R}^+(G) = 0$ and $G^{\ast }$ is conditionally 
$\sigma $-complete then $L$ is basically disconnected.
\end{proposition}

\begin{proof}
Consider a cozero element $a\in L$, say $a=f\left( 0,\infty\right) $ for $%
f\in C^{+} L$. By replacing $f$ by $f\wedge1$, we may assume that $f\in 
\mathcal{R}^{\ast}L=G^{\ast}$. Define the sequence $\left\{ g_{n}\right\} $
in $G^{\ast}$ by setting 
\begin{equation*}
g_{n}\equiv nf\wedge1,\;n\in\mathbb{N},
\end{equation*}
and let $g\in G^{\ast}$ be such that $g=\bigvee_{\mathbb{N}}g_{n}$.

We aim to show that $g$ is a component of $1$, i.e., that $\left( 1-g\right)
\wedge g=0$, by means of several claims. We first claim that $\left(
1-g\right) \wedge\left( nf-1\right) ^{+} =0$ for all $n$. For%
\begin{equation*}
g={\bigvee\nolimits}_{\mathbb{N}}g_{n}\Longrightarrow1-g = {%
\bigwedge\nolimits}_{\mathbb{N}}\left(1-g_{n}\right) = {\bigwedge\nolimits}_{%
\mathbb{N}}\left( 1-nf\right) ^{+} ,
\end{equation*}
and, since $\left( 1-nf\right) ^{+} \wedge\left( nf-1\right) ^{+} =0$, 
\begin{equation*}
\left( 1-g\right) \wedge\left( nf-1\right) ^{+} \leq\left( 1-nf\right) ^{+}
\wedge\left( nf-1\right) ^{+} =0.
\end{equation*}
We next claim that $\left( 1-g\right) \wedge f=0$. For if not, then $%
x\equiv\left( 1-g\right) \wedge f\wedge1>0$. Since $G$ is archimedean, there
exists $k\in\mathbb{N}$ such that $kx\nleq1$; let $k$ be the least such
integer. Then 
\begin{equation*}
0<\left( kx-1\right) ^{+} \leq\left( kf-1\right) ^{+} \Longrightarrow\left(
kx-1\right) ^{+} \wedge\left( 1-g\right) =0.
\end{equation*}
But%
\begin{equation*}
\left( k-1\right) x\leq1\Longrightarrow\left( kx-1\right) ^{+} \leq
x\leq\left( 1-g\right) ,
\end{equation*}
a contradiction. It now follows that $\left( 1-g\right) \wedge nf=0$ for all 
$n$, hence $\left( 1-g\right) \wedge g_{n}=0$ for all $n$. Upon recalling a
basic fact about $\ell$-groups, namely that if $g=\bigvee_{\mathbb{N}}g_{n}$
for $\left\{ g_{n}\right\} \subseteq G^{+} $ and if $a\wedge g_{n}=0$ for
all $n$ then $a\wedge g=0$, we reach the desired conclusion: $g\wedge\left(
1-g\right) =0$.

The basic disconnectivity of $L$ follows immediately from the fact that $g$
is a component of $1$, for 
\begin{equation*}
g\wedge\left( 1-g\right) =0\Longrightarrow g\vee\left( 1-g\right) =g+\left(
1-g\right) =1,
\end{equation*}
hence $\coz g\vee\coz \left( 1-g\right) =\coz 1=\top$, which is to say that $%
a^{\ast\ast}\vee a^{\ast }=\top$.
\end{proof}

The proof of Theorem \ref{Thm:1} is completed by Proposition \ref{Prop:2}, 
which requires two simple lemmas, the first of which is folklore. 
We say that an $\ell$-subgroup 
$H \leq G$ is \emph{order dense} in $G$ if for every $0 < g \in G$ there is some
$h \in H$ such that $0 < h \leq g$

\begin{lemma}\label{Lem:7}
Suppose $G$ is an order dense $\ell$-subgroup of $H$.
\begin{enumerate}
\item
Then suprema and infima in $G$ and $H$ agree, and

\item
$h = \bigvee \downset{h}_G$ for all $h \in H^+$.
\end{enumerate}
\end{lemma}

\begin{proof}
(1) Suppose $\bigvee A = g_0$ for some $A \subseteq G^+$ and 
$g_0 \in G^+$,
but that $A \leq h < g_0$ for some $h \in H$. Find $g_1 \in G$ such that
$0 < g_1 \leq g_0 - h$. Since $g_0 - g_1 < g_0$ there is some $g	\in A$
for which $g \nleq g_0 - g_1$.  But this flies in the face of the fact that 
$g + g_1 \leq h + g_1 \leq g_0$. We conclude that $\bigvee A = g_0$ in $H$.

(2) Given $h_0 \in H^+$, let $A \equiv \downset{h_0}_{G^+}$ 
and suppose for
the sake of argument that $A \leq h_1 < h_0$ for some $h_1 \in H^+$. 
Then find $g_0 \in G$ such 
that $0 < g_0 \leq h_0 - h_1$.  But for any $g \in A$ we have 
$g + g_0 \leq h_1 + g_0 \leq h_0$, i.e., $A + g_0 \subseteq A$.
It follows that $ng_0 \in A$ for all $n$, which is to say that 
$ng_0 \leq h_0$ for all $n$, a violation of the archimedean property
of $H$.
\end{proof}

\begin{lemma}
For $h,k \in \mathcal{R}^+L$ and $0 \leq q \in \mathbb{Q}$, 
if $\coz k \leq h(q, \infty)$ and 
$k \leq q$ then $k \leq h$.
\end{lemma}

\begin{proof}
By Lemma \ref{Lem:3}(2) it is sufficient to show that 
for any $r \in \mathbb{R}$, 
\[
h(r,\infty) \geq k(r,\infty)
 = (k \wedge q)(r,\infty) 
= k(r,\infty) \wedge q(r,\infty)
=\begin{cases} k(r,\infty) & \text{if $r < q$}\\
                        \bot &\text{if $r \geq q$}\end{cases}. 
\]
But this is clear, for if $0 \leq r < q$ then 
$k(r,\infty) \leq k(0,\infty) 
= \coz k \leq h(q,\infty) \leq h(r,\infty)$.\qedhere
\end{proof}

%

We remind the reader that a cozero element $a$ of a Lindel\"{o}f frame $L$
is Lindel\"{o}f, i.e., $a = \bigvee A$ implies $a = \bigvee A_0$ for some
countable subset $A_0 \subseteq A$.
 
\begin{lemma}\label{Lem:5}
If $\bigwedge \mathbb{R}^+(G) = 0$ then every element of $\mathcal{R}L$ 
is the join of a countable subset of $\widehat{G}^*$. 
\end{lemma}

\begin{proof}
Given $0 < h \in \mathcal{R}^+L$ and $q \in \mathbb{R}^+(G)$, 
$h(q,\infty) = \coz(h - q)^+ $ is a cozero element of $L$ and is therefore
Lindel\"{o}f.  Since $\mu_G : G \to \mathcal{R}L$ is cozero dense, this 
element is the join of those of the form $\coz \widehat{g}$, $g \in G^+$.
Let $G_q$ be a countable subset of $G^+$ such that 
$h(q,\infty) = \bigvee_{G_q} \coz \widehat{g}$.  
By suitably restricting and augmenting
$G_q$, we may assume that $g \leq q$ for all $g \in G_q$, and that 
$g \in G_q$ implies $mg \wedge q \in G_q$ for all integers $m$.
Finally, let $G_0 \equiv \bigcup_{0<q\in R}G_q$ for some countable 
dense subset $R \subseteq \mathbb{R}^+(G)$.

We claim that $h = \bigvee G_0$. If not then $G_0 \leq k < h$ for some 
$k \in \mathbb{R}L$, so that by Lemma \ref{Lem:3}(2) there is some 
$q \in R$ such that $k(q,\infty) < h(q,\infty)$. More is true; there must
be some $s > q$ in $R$ for which $a \equiv k(-\infty,s) \wedge h(s,\infty) > \bot$, 
for otherwise $h(s,\infty) \leq k(-\infty,s)^*$ for all $s > q$ would imply 
\[
h(q,\infty) =
\bigvee_{q < s}h(s,\infty) \leq \bigvee_{q < s}k(-\infty,s)^* 
= k(q,\infty),
\]
contrary to assumption. Now $h(s,\infty) = \bigvee_{G_s} \coz \widehat{g}$, 
so we may find $g \in G_s$ such that $\coz g \wedge a > \bot$. 
Since $\coz g = \bigvee_m\coz(mg - k)^+$ by Corollary \ref{Cor:1}(5),
there exists some integer $m$ for which $\coz(mg - k)^+ \wedge a > \bot$.
In sum, we have arranged that 
\begin{align*}
\coz(mg\wedge s -k)^+ 
&= \coz ((mg - k)^+\wedge (s - k)^+)
= \coz (mg - k)^+\wedge \coz(s - k)^+ \\
&= \coz (mg - k)^+ \wedge k(-\infty,s) > \bot.
\end{align*}
But $g \in G_s$ implies $mg\wedge s \in G_s$, hence $mg \wedge s \leq k$,
contrary to the information displayed above. This completes the proof of the claim
and the lemma.   

\end{proof}

\begin{proposition}\label{Prop:2} 
If $\bigwedge \mathbb{R}^+(G) = 0$ and $G^*$ is conditionally 
$\sigma$-complete then $\widehat{G}^* = \mathcal{R}^*L$.
\end{proposition}

\begin{proof}
Given $0 < h \in \mathcal{R}^*L$, we know from Lemma \ref{Lem:5}
that $h=  \bigvee A$ for some countable subset
$A \subseteq \downset{h}_{\widehat{G}}$. 
Since $h$ is bounded by some multiple of $1$, so 
is $A$. By virtue of the conditional $\sigma$-completeness of $G$, 
$A$ has a supremum $g_0$ in $G$.  Finally, $\widehat{g}_0 = h$ by 
Lemmas \ref{Lem:7} and \ref{Lem:5}. 
\end{proof}

\begin{corollary}
\label{Cor:3}A conditionally $\sigma $-complete $\mathbf{W}$-object which
contains arbitrarily small positive multiples of $1$ contains all real
multiples of $1$. That is, $\bigwedge\mathbb{R}^+(G) = 0$ 
implies $\mathbb{R}^{+}\left( G\right) =\mathbb{R}^{+}$.
\end{corollary}

The proof of Theorem \ref{Thm:1} is complete.

It is worthwhile to restate Theorem \ref{Thm:1} in the language of regular 
$\sigma$-frames. This is always possible, since the fact that $L$ is 
Lindel\"{o}f means that $L$ is isomorphic to $\mathcal{H}\coz L$.

\begin{theorem}\label{Thm:8}
Let $G$ be a $\mathbf{W}$-object with Madden frame $L$. Then conditions (1)
and (2) together are equivalent to conditions (3) and (4).

\begin{enumerate}
\item 
$\bigwedge \mathbb{R}^+(G) = 0$ .

\item 
$G^{\ast }$ is conditionally complete ($\sigma $-complete).

\item 
$L$ is extremally disconnected (basically disconnected).

\item[(3')] $L$ is isomorphic to $\mathcal{H}A$ for some regular $\sigma $%
-frame $A$ such that for all $C\subseteq A$ there exists a complemented
element $b\in A$ with 
\begin{equation*}
\forall~d\in A~\left( \forall~c\in C~\left( d\wedge c=\bot\right)
\Longleftrightarrow d\leq b\right) .
\end{equation*}
($L$ is isomorphic to $\mathcal{H}A$ for some regular $\sigma$-frame $A$
such that for all $c\in A$ there exists a complemented element $b\in A$ with 
\begin{equation*}
\forall~d\in A~\left( d\wedge c=\bot\Longleftrightarrow d\leq b\right) .)
\end{equation*}

\item[(4)] The Madden representation carries $G^{\ast}$ onto $\mathcal{R}%
^{\ast}L$.
\end{enumerate}
\end{theorem}

\subsection{$P$-frames and boolean frames\label{Subsec:4}}

Proposition \ref{Prop:16} holds that conditional pointwise completeness is
stronger than conditional completeness. In view of Theorem \ref{Thm:1}, 
then, the question naturally arises as to what condition on $\mathcal{M}G$
is equivalent to the conditional pointwise completeness of a 
$\mathbf{W}$-object $G$. The answer is that $\mathcal{M}G$ must be
boolean in order for $G$ to be conditionally pointwise complete,
and $\mathcal{M}G$ must be a $P$-frame in order for $G$ to be conditionally
pointwise $\sigma $-complete.

\begin{theorem}
\label{Thm:3}Let $G$ be a $\mathbf{W}$-object with Madden frame $L$. Then
conditions (1) and (2) together are equivalent to conditions (3) and (4).

\begin{enumerate}
\item $G$ contains arbitrarily small positive multiples of $1$.

\item $G^{\ast }$ is conditionally pointwise complete (conditionally
pointwise $\sigma $-complete).

\item $L$ is boolean (a $P$-frame).

\item The Madden representation carries $G^{\ast}$ onto $\mathcal{R}^{\ast}L$%
.
\end{enumerate}
\end{theorem}

\begin{proof}
We first prove the countable version of this theorem. Suppose $L$ is a $P$%
-frame, identify $G$ with its Madden representation $\widehat{G}\leq\mathcal{%
R}L$, and suppose $G^{\ast}=\mathcal{R}^{\ast}L$. Then $G$ certainly
satisfies (1); in order to verify that $G^{\ast}$ is pointwise $\sigma$%
-complete, consider a countable subset $F\subseteq G$ with upper bound $g\in
G^{\ast}$. Define a function $f_{0}$ on the right rays by the rule 
\begin{equation*}
f_{0}\left( r,\infty\right) \equiv{\bigvee\nolimits}_{F}f\left(
r,\infty\right), \quad r\in\mathbb{R}.
\end{equation*}
We claim that $f_{0}$ extends to a unique member of $\mathcal{R}L$, which
must then lie in $G^{\ast}$ by virtue of its convexity, since clearly $f\leq
f_{0}\leq g$ by Lemma \ref{Lem:3}(2). To establish this claim we need only
check the three hypotheses of Lemma \ref{Lem:1}. The first hypothesis
clearly holds, since a complemented element of any frame is rather below
itself. To verify the second, simply observe that 
\begin{equation*}
\bigvee_{s>r}f_{0}\left( s,\infty\right) =\bigvee_{s>r}\bigvee_{F}f\left(
s,\infty\right) =\bigvee_{F}\bigvee_{s>r}f\left( s,\infty\right)
=\bigvee_{F}f\left( r,\infty\right) =f_{0}\left( r,\infty\right) .
\end{equation*}
To verify the third hypothesis, note that 
\begin{equation*}
{\bigvee\nolimits}_{\mathbb{R}}f_{0}\left( r,\infty\right) = {%
\bigvee\nolimits}_{\mathbb{R}}{\bigvee\nolimits}_{F}f\left( r,\infty\right)
= {\bigvee\nolimits}_{F}{\bigvee\nolimits}_{\mathbb{R}}f\left(r,\infty%
\right) = \top.
\end{equation*}
And, since $f_{0}\left( r,\infty\right) =\bigvee_{F}f\left( r,\infty \right)
\leq g\left( r,\infty\right) $ for all $r\in\mathbb{R}$, 
\begin{equation*}
{\bigvee\nolimits}_{\mathbb{R}}f_{0}\left( r,\infty\right) ^{\ast}\geq {%
\bigvee\nolimits}_{\mathbb{R}}g\left( r,\infty\right)^{\ast} \geq {%
\bigvee\nolimits}_{\mathbb{R}}g\left(-\infty,r\right) =\top.
\end{equation*}
The second inequality holds because $g\left( -\infty,r\right) \wedge g\left(
r,\infty\right) =g\left( \emptyset\right) =\bot$.

Now suppose that $G$ satisfies (1) and (2), and again identify it with its
Madden representation. Then $G^{\ast}$ is $\sigma$-complete by Remark \ref%
{Rem:1}(5), so that Theorem \ref{Thm:1} allows us to conclude that $G^{\ast}=%
\mathcal{R}^{\ast}L$. To verify (3), consider a cozero element $a\in L$, say 
$a=\coz f$ for some $f\in\mathcal{R}^{+} L$. By replacing $f$ by $f\wedge1$
if necessary, we may assume that $1\geq f\in G^{\ast}$. Define the sequence $%
\left\{ g_{n}\right\} $ in $G^{\ast}$ by setting 
\begin{equation*}
g_{n}\equiv nf\wedge1,\;n\in\mathbb{N},
\end{equation*}
and let $g\in G^{\ast}$ be such that $g = \bigvee^\bullet _{\mathbb{N}}g_{n}$%
. Since the $g_{n}$'s here are defined exactly as in the proof of Theorem %
\ref{Thm:1}, and since $g=\bigvee_{\mathbb{N}}g_{n}$, the argument given
there applies here, and shows that 
\begin{equation*}
\coz g\vee\coz\left( 1-g\right) = \coz 1=\top.
\end{equation*}
But when we observe that $\coz g=g\left( 0,\infty\right) =a$ and $\coz\left(
1-g\right) =a^{\ast}$, we come to the desired conclusion: $a\vee
a^{\ast}=\top$.

It remains to prove the version of the theorem in which the pointwise joins
are of unrestricted cardinality. For the most part the argument goes along
the lines of the countable case. The only significant departure is the
implication from (1) and (2) to (3). So suppose $G$ contains arbitrarily
small positive multiples of $1$ and that $G^{\ast}$ is pointwise complete,
and consider an arbitrary element $a_{0}\in L$. Express $a_{0}$ in the form $%
\bigvee_{I}a_{i}$, where $\left\{ a_{i}:i\in I\right\} $ is the set of
cozero elements below $a_{0}$. For each $i\in I$ let $g_{i}$ be the
characteristic function of $a_{i}$, i.e., 
\begin{equation*}
g_{i}\left( U\right) =\left\{ 
\begin{array}{lll}
\bot & \text{if} & 0,1\notin U \\ 
a_{i} & \text{if} & 0\notin U\ni1 \\ 
a_{i}^{\ast} & \text{if} & 1\notin U\ni0 \\ 
\top & \text{if} & 0,1\in U%
\end{array}
\right. .
\end{equation*}
These functions lie between $0$ and $1$ and hence are in $G^{\ast}$. Let $%
g_{0}\equiv \bigvee^\bullet _{I}g_{i}$. By inspection one sees that%
\begin{equation*}
g_{0}\left( U\right) =\left\{ 
\begin{array}{lll}
\top & \text{if} & r<0 \\ 
a_{0} & \text{if} & 0\leq r<1 \\ 
\bot & \text{if} & 1\leq r%
\end{array}
\right. ,
\end{equation*}
the characteristic function of $a_{0}$. But since $a_{0}=g_{0}\left(
3/4,\infty\right) \prec g_{0}\left( 1/4,\infty\right) =a_{0}$, it follows
that $a_{0}$ is complemented. The shows that $L$ is boolean, and completes
the proof.
\end{proof}

A quotient of a $P$-frame need not be a $P$-frame (\cite%
{BallWaltersZenk:2010}). However, a $C$-quotient of a $P$-frame is clearly a 
$P$-frame, for a $C$-quotient $f:L\rightarrow M$ is coz-onto, meaning every
cozero element of $M$ is the image under $f$ of a cozero element of $L$.
Since the cozero elements of $L$ are complemented, so are their images. An
alternative argument can be made using Theorem \ref{Thm:3}.

\begin{corollary}
\label{Cor:6}A $C$-quotient of a $P$-frame is a $P$-frame.
\end{corollary}

\begin{proof}
A $C$-quotient map $f:L\rightarrow M$ induces a $\mathbf{W}$-surjection $%
\mathcal{R}f:\mathcal{R}L\rightarrow \mathcal{R}M$. If $L$ is a $P$-frame
then conditions (3) and (4) of Theorem \ref{Thm:3} hold, and therefore
conditions (1) and (2) are true of $\mathcal{R}L$. But the latter two
conditions are clearly inherited by any quotient of $\mathcal{R}L$, and
therefore are true of $\mathcal{R}M$. A second application of the theorem
gives the desired conclusion.
\end{proof}

\section{Unconditional pointwise completeness\label{Sec:6}}

In this section we define and analyze the ultimate, or unconditional form of
pointwise completeness. This naturally raises the question of precisely what
unconditional pointwise completeness ought to mean. Our definition comes in
Subsection \ref{Subsec:9}, but it requires a digression to review essential
extensions in Subsection \ref{Subsec:7} and cuts in Subsection \ref{Subsec:8}%
. The reader may wish to skip this material upon a first reading, returning
to it as necessary.

\subsection{Essential extensions and complete embeddings}

\label{Subsec:7} In this subsection we recall the basic facts concerning
essential extensions in $\mathbf{W}$. We do so not only because we will make
use of these facts in the sequel, but also for the reader\rq{}s convenience,
for these extensions appear in the literature under various names and with
various definitions. Because this material is well known (see, e.g., \cite%
{BigardKeimelWolfenstein:1977}), we offer here only hints of proofs.

Recall that the \emph{booleanization} of a frame $M$ is the frame map 
\begin{equation*}
b_M :M\rightarrow M^{\ast \ast }= (a\mapsto a^{\ast \ast }).
\end{equation*}
In spatial terms, this is the map which sends an open set to its
regularization, i.e., the smallest regular open subset containing it.

\begin{lemma}
\label{Lem:9}The following are equivalent for an extension $G\leq H$ in $%
\mathbf{W}$.

\begin{enumerate}
\item The embedding $G\rightarrow H$ is an essential monomorphism, i.e., any
morphism out of $H$ whose restriction to $G$ is one-one is also one-one on $%
H $.

\item Every nontrivial $\mathbf{W}$-kernel of $H$ meets $G$ nontrivially.

\item Every nontrivial polar of $H$ meets $G$ nontrivially.

\item $G$ is large in $H$, i.e., every nontrivial convex $\ell $-subgroup of 
$H$ meets $G$ nontrivially.

\item If $H$ is divisible then these conditions are equivalent to $G$ being
order dense in $H$, i.e., for every $0<h\in H$ there exists some $0<g\in G$
such that $g\leq h$.

\item The frame map $f:L\equiv \mathcal{M}G\rightarrow \mathcal{M}H\equiv M$
which realizes the extension $G\leq H$ \enquote*{drops} to an isomorphism of
the booleanizations. 
\begin{figure}[h]
\setlength{\unitlength}{4pt}
\par
\begin{center}
\begin{picture}(12,10)(0,2)
\small
\put(0,12){\makebox(0,0){$L$}}
\put(12,12){\makebox(0,0){$M$}}
\put(12,0){\makebox(0,0){$M^{**}$}}
\put(0,0){\makebox(0,0){$L^{**}$}}
\put(2,12){\vector(1,0){8}}
\put(0,10){\vector(0,-1){8}}
\put(12,10){\vector(0,-1){8}}
\put(2.5,0){\vector(1,0){6}}
\put(6,13.5){\makebox(0,0){$f$}}
\put(-2,6){\makebox(0,0){$b_L$}}
\put(14.5,6){\makebox(0,0){$b_M$}}
\end{picture}
\end{center}
\end{figure}
\end{enumerate}
\end{lemma}

\begin{proof}
The equivalence of (1) with (2) is evident, and the implication from (2) to
(3) is a consequence of the fact that polars are $\mathbf{W}$-kernels. To
show that (3) implies (4), one shows that $K^{\bot\bot} \cap G = 0$ for any
convex $\ell$-subgroup $K \leq H$ such that $K \cap G = 0$. And (4) clearly
implies (2).

Part (3) can be interpreted as saying that the extension is essential iff
the polars of $G$ and $H$ are in bijective correspondence by intersection.
But the polars of any $\mathbf{W}$-object are in bijective correspondence
with the elements of its booleanization via 
\begin{eqnarray*}
a^{\ast \ast } &\longrightarrow &\left\{ g\in G:\coz\left\vert g\right\vert
\leq a^{\ast \ast }\right\} \\
\bigvee\nolimits_{P}\coz g &\longleftarrow &P
\end{eqnarray*}
This observation can be readily converted into a proof that (6) is
equivalent to the other conditions.
\end{proof}

From Lemma \ref{Lem:9}(6) we see that for any frame $L$ the dual $\mathcal{R}%
b_{L}$ of the booleanization map provides an essential extension $\mathcal{R}%
L\rightarrow \mathcal{R}L^{\ast \ast }$. It is this embedding which is meant
whenever we write $\mathcal{R}L\leq \mathcal{R}L^{\ast \ast }$.

We say that an object is \emph{essentially complete} if it has no proper
essential extensions. A \emph{maximal essential extension} of $G$ is an
essential extension $G \leq H$ such that $H$ is essentially complete.

\begin{proposition}
\label{Prop:17}

\begin{enumerate}
\item A $\mathbf{W}$-object is essentially complete iff it is of the form $%
\mathcal{R}L$ for a boolean frame $L$.

\item Every $\mathbf{W}$-object $G$ has a maximal essential extension,
namely $G \leq \mathcal{R}L \leq \mathcal{R}L^{**}$.

\item Any two maximal essential extensions of $G$ are isomorphic over $G$.

\item Let $G \leq H$ be a maximal essential extension. Then an arbitrary
extension $G \leq K$ is essential iff $K$ is isomorphic over $G$ to an $\ell$%
-subgroup of $H$
\end{enumerate}
\end{proposition}

\begin{proof}
(1) is a consequence of the fact that $G\leq \mathcal{R}L\leq \mathcal{R}%
L^{\ast \ast }$ is an essential extension which is an isomorphism iff $G=%
\mathcal{R}L=RL^{\ast \ast }$. (2) is Proposition 2.1 of \cite%
{BanaschewskiHager:2013}. The rest is due to Conrad from his seminal article 
\cite{Conrad:1970}.
\end{proof}


\cite{BanaschewskiHager:2013} provides an analysis of maximal essential
extensions in categories related to $\mathbf{W}$. See also \cite%
{BanaschewskiHager:2013(2)} for a closely related analysis in the context of
completely regular frames.

Essential extensions take their importance here from the fact that any
extension may be \enquote*{reduced} to an essential extension by passage to
an appropriate quotient. This is Lemma \ref{Lem:10}, which involves an
attribute weaker than essentiality. Recall that an injective homomorphism $%
\tau :H\rightarrow K$ is said to be \emph{complete} if it preserves all
suprema and infima that exist in $H$. The following is folklore; see, e.g., 
\cite{Darnel:1994}.

\begin{lemma}
\label{Lem:15}An essential injection is complete.
\end{lemma}

\begin{proof}
Let $G\leq H$ be an essential extension, and let $\bigvee Z=g$ for some
subset $Z\subseteq G^{+}$ and element $g\in G^{+}$. If $G$ fails to be the
supremum of $Z$ in $H$, it is only because there is some $h\in H$ such that $%
Z\leq h<g$. Now the convex $\ell $-subgroup of $H$ generated by $g-h$ is
nontrival, and by Lemma \ref{Lem:9}(4) contains some $0<g^{\prime }\in G$,
say $g^{\prime }\leq n\left( g-h\right) $ for a positive integer $n$. This
rearranges to 
\begin{equation*}
ng>ng-g^{\prime }\geq nh\geq nz\;\;\text{for all }z\in Z\text{.}
\end{equation*}%
But $\bigvee Z=g$ implies $\bigvee_{Z}nz=ng$ in $G$, and this contradicts
the displayed condition.
\end{proof}

\begin{lemma}
\label{Lem:10}For any injective homomorphism $\gamma :G\rightarrow H$ there
is a surjective homomorphism $\tau :H\rightarrow K$ such that $\tau \circ
\gamma $ is an essential injection. Moreover, $\tau $ may be chosen to be
complete.
\end{lemma}

\begin{proof}
We claim that the family $\mathcal{Q}$ of polars $Q$ such that $Q\cap \gamma
\lbrack G]=0$ contains maximal elements. For if $\mathcal{C}$ is a nonempty
chain in $\mathcal{Q}$ then, since $\gamma \lbrack G]\cap \bigcup \mathcal{C}%
=0$ and $\bigcup \mathcal{C}$ is convex, $\gamma \lbrack G]\subseteq
(\bigcup \mathcal{C})^{\bot }$, hence $\gamma \lbrack G]\cap (\bigcup 
\mathcal{C})^{\bot \bot }=0$ and so $(\bigcup \mathcal{C})^{\bot \bot }\in 
\mathcal{Q}$. If we take $\tau $ to be the quotient map $H\rightarrow H/R$
for some polar $R$ maximal in $\mathcal{Q}$ then it is clear that part (3)
of Lemma \ref{Lem:9} is satisfied by $\tau \circ \gamma $. And $\tau $ is
complete because $Q$ is order closed.
\end{proof}

\begin{lemma}
\label{Lem:14}If a subset $Z\subseteq G$ has a supremum in some extension
then it has a supremum in some essential extension.
\end{lemma}

\begin{proof}
Let $G\leq H\,$be an extension such that $\bigvee Z=h$ in $H$, and let $\tau
:H\rightarrow K$ be the complete surjection of Lemma \ref{Lem:10} such that $%
\tau $ is one-one on $Z$ and $K$ is an essential extension of $\tau \left[ Z%
\right] $. Then $\bigvee \tau \left[ Z\right] =\tau \left( h\right) $ in $K$
since $\tau $ is complete. Identifying $Z$ with its image under $\tau $
provides the desired extension.
\end{proof}


\subsection{Mobile downsets and cuts\label{Subsec:8}}

Many completion results are based on the technique of adjoining to $G$ a
supremum for each downset of a particular type. The downsets in play,
usually called \emph{cuts}, depend on the sort of completeness desired. The
broadest notion of a cut was introduced in Section 4 of \cite{Ball:1980}.

\begin{definition}
\label{Def:7}A downset $Z \subseteq G$ is called a \emph{cut} if it has a
supremum in some extension of $G$.
\end{definition}

Observe that by Lemma \ref{Lem:14}, a downset $Z \subseteq G^+$ is a cut iff
it has a supremum in some essential extension of $G$.

Definition \ref{Def:7} is sufficiently opaque as to appear useless, but its
utility is restored by Theorem \ref{Thm:5}, which gives a working criterion
for a subset $Z \subseteq G^+$ to be a cut. That criterion involves the
inability of the subset to remain stationary under addition by a positive
element.


%
%
%
%

\begin{definition}
A downset $Z\subseteq G$ is said to be \emph{mobile} if $Z+g\nsubseteq Z$
for all $0<g\in G$.
\end{definition}

\begin{observations}
\label{Obs:1} Let $Z$ be a downset in $G$.

\begin{enumerate}
\item $Z$ is mobile iff there is no $0 < g \in G$ for which $Z + G(g)
\subseteq Z$, where $G(g)$ designates the convex $\ell$-subgroup of $G$
generated by $g$.

\item $Z$ is mobile iff it is not the union of cosets of some nontrivial
convex $\ell$-subgroup of $G$.

\end{enumerate}
\end{observations}

\begin{proof}
If $Z + g \subseteq Z$ then $Z + ng \subseteq Z$ for all $n$, hence $Z + k
\subseteq Z$ for all $k$ such that $|k| \leq ng$ for some $n$.
\end{proof}

The next proposition hints at why mobile downsets are relevant to our
investigation, for it shows that two types of subsets which may have
pointwise joins are mobile

\begin{proposition}
\label{Lem:16}

\begin{enumerate}
\item The downset $\left\downarrow g_{0}\wedge n:n\in \mathbb{N}%
\right\downarrow$ generated by the truncates of an element $g_{0}\in G^{+}$ is
mobile.

\item A nonempty bounded downset is mobile.
\end{enumerate}
\end{proposition}

\begin{proof}
(2) Suppose the downset $\emptyset \neq Z\subseteq G$ is bounded above by $%
g_{0}$, and suppose for the sake of argument that $Z+g\subseteq Z$ for some $%
0<g\in G$. We may assume that $0\in Z$, for we may always replace $Z$ by $%
Z-z_{0}$ and $g_{0}$ by $g_{0}-z_{0}$, where $z_{0}$ is any member of $Z$.
But then $ng\in Z$ for all positive integers $n$, with the result that $%
ng\leq g_{0}$ for all $n$, a violation of the archimedean property of $G$.

(1) Let $Z\equiv \left\downarrow g_{0}\wedge n:n\in \mathbb{N}%
\right\downarrow $ be the set of lower bounds of the truncates of $g_{0}\in
G^{+}$, and suppose for the sake of argument that $Z+g\subseteq Z$ for some $%
0<g\in G$. Then by the archimedean property there is a positive integer $m$
such that $mg\nleq g_{0}$. It follows that $mg\nleq n\wedge g_{0}$ for any
positive integer $n$, which is to say that $mg \notin Z$, contrary to
Observation \ref{Obs:1}(1).
\end{proof}

\begin{theorem}
\label{Thm:5}The following are equivalent for a downset $Z\subseteq G$.

\begin{enumerate}
\item $Z$ is a cut in $G$.

\item $G$ has an essential extension $H$ containing an element $h$ such that 
$\bigvee Z=h$ in $H$.

\item $G$ is completely embedded in an extension $H$ containing an element $%
h $ such that $\bigvee Z=h$ in $H$.

\item $Z$ is not a union of cosets of a nontrivial convex $\ell$-subgroup of 
$G$.

\item $Z$ is mobile.
\end{enumerate}
\end{theorem}

\begin{proof}
The equivalence of (1) and (2) is Lemma \ref{Lem:14}, the implication from
(2) to (3) is Lemma \ref{Lem:15}, and the implication from (3) to (1) is
trivial. The equivalence of (3) and (5) is Proposition 4.5 of \cite%
{Ball:1989}, and the equivalence of (4) and (5) is Observation \ref{Obs:1}%
(2) .
\end{proof}

Our main Theorem \ref{Thm:4} requires a technical lemma.

\begin{lemma}\label{Lem:2}
Let $G=\mathcal{R}L$ for a $P$-frame $L$, and let $G\leq H$ be
its maximal essential extension (Proposition \ref{Prop:17}). If all of the
truncates of an element $h \in H^+$ lie in $G$ then $h$ lies in $G$
\end{lemma}

\begin{proof}
Here $G=\mathcal{R}L\leq \mathcal{R}L^{\ast \ast }=H$, where the embedding $%
\mathcal{R}L\rightarrow \mathcal{R}L^{\ast \ast }$ is provided by $\mathcal{R%
}b_{L}=\left( g\longmapsto b_{L}\circ g\right) $. Suppose $\left\{
g_{n}\right\} \subseteq G^{+}$. The condition that $g_{n+1}\wedge n=g_{n}$
for all $n$ clearly holds in $H$ iff it holds in $G$, and the condition that 
$\bigvee_{n}g_{n}\left( -\infty ,n\right) =\top $ implies that 
\begin{equation*}
\top =b_{L}\left( \top \right) =b_{L}\left( \bigvee_{n}g\left( -\infty
,n\right) \right) =\bigvee_{n}b_{L}\circ g\left( -\infty ,n\right)
\end{equation*}%
because $b_{L}$ is a frame morphism. We must demonstrate the converse, i.e.,
that the displayed condition implies that that $\bigvee_{n}g_{n}\left(
-\infty ,n\right) =\top $.

So assume that $\bigvee_{n}b_{L}\circ g_{n}\left( -\infty ,n\right) =\top $
holds in $L^{\ast \ast }$, i.e., that $\left( \bigvee_{n}g_{n}\left( -\infty
,n\right) ^{\ast \ast }\right) ^{\ast \ast }=\top $ holds in $L$. Note that $%
g_{n}\left( -\infty ,n\right) $, being a cozero element of a $P$-frame, is
complemented, i.e., $g_{n}\left( -\infty ,n\right) ^{\ast \ast }=g_{n}\left(
-\infty ,n\right) $. Also note that, since the inclusion $\coz L\rightarrow L
$ is a $\sigma $-frame homomorphism, the supremum $\bigvee_{n}g_{n}\left(
-\infty ,n\right) $ in $L$ agrees with its supremum in $\coz L$. But the
latter is a cozero and is therefore complemented, so that we get $\left(
\bigvee_{n}g_{n}\left( -\infty ,n\right) \right) ^{\ast \ast
}=\bigvee_{n}g_{n}\left( -\infty ,n\right) $.
\end{proof}

The hypotheses of the preceding lemma are more generous than necessary, so 
we digress briefly to tighten it up in the light of 
the analysis conducted by the second author in \cite{Hager:2013}.  
We refer the interested reader to that article for terminology and notation 
otherwise undefined here, and omit the details of proof.

\begin{theorem}\label{Thm:9}
The following are equivalent for a $\mathbf{W}$-object $G$
with maximal essential extension $G \leq H$.
\begin{enumerate}
\item
Every element $h \in H^+$ which has all its truncates in $G$ must 
lie in $G$.

\item
Every element $h \in D^+(YG)$ with all its truncates in $G$ must lie in $G$.

\item
$G$ is *-maximum, i.e., $G$ contains a copy of every $\mathbf{W}$-object
with the same bounded part as $G$.

\end{enumerate}
\end{theorem}

\subsection{Pointwise completeness \label{Subsec:9}}

\begin{definition}
A $\mathbf{W}$-object is pointwise complete ($\sigma$-complete) if every
(countably generated) cut in $G$ has a pointwise join in $G.$
\end{definition}

\begin{theorem}
\label{Thm:4}The following are equivalent for a $\mathbf{W}$-object $G$.

\begin{enumerate}
\item $G$ is pointwise complete ($\sigma $-complete).

\item Every (countably generated) mobile downset of $G$ has a pointwise join
in $G$.

\item $G$ is of the form $\mathcal{R}L$ for a boolean frame ($P$-frame) $L$.
\end{enumerate}
\end{theorem}

\begin{proof}
The equivalence of (1) with (2) follows from Theorem \ref{Thm:5}. If $G$
satisfies (2) then by Lemma \ref{Lem:16}(2) every bounded (countable) subset
of $G^{+}$ has a pointwise join in $G$, hence by Theorem \ref{Thm:3} $G$ is
a subobject of $\mathcal{R}L$, $L$ boolean (a $P$-frame), and $G$ contains $%
\mathcal{R}^{\ast }L$. That $G$ is actually all of $\mathcal{R}L$ follows
from Proposition \ref{Prop:3}.

Let us show that (3) implies (1) when $G$ is of the form $\mathcal{R}L$ for
some $P$-frame $L$. Let $Z$ be a countable subset of $G^{+}$ such that $%
\bigvee Z=h_{0}$ in some extension $G\leq H$. By Lemma \ref{Lem:14} we may
assume this extension to be essential, and, in fact, it does no harm to
assume that $H$ is the maximal essential extension of $G$. That is because
any essential extension of $G$ is isomorphic to an $\ell$-subgroup of $H$
containing $G$, and suprema in all these extensions agree by Lemma \ref%
{Lem:15}.

Since $G$ is conditionally pointwise $\sigma$-complete by Theorem \ref{Thm:3}%
, for each $n$ the subset $Z \wedge n$ has a pointwise supremum $g_n \in G$.
Of course, the same subset has supremum $h_0 \wedge n$ in $H$, and the two
suprema coincide by Proposition \ref{Prop:7}. Lemma \ref{Lem:2} then implies
that $h_0 \in G$. Since $h_0$ is the pointwise join of the $g_n$\rq{}s by
Proposition \ref{Prop:3}, and in light of the fact that each $g_n$ is the
pointwise join of $Z \wedge n$, it follows that $\bigvee^\bullet Z = h_0 \in
G$.

It remains to show that (3) implies (1) in case $G=\mathcal{R}L$ for some
boolean frame $L$. Let $Z\subseteq G^{+}$ be such that $\bigvee Z=h_{0}$ in
some extension $G\leq H$. By Lemma \ref{Lem:14} again, we may assume this
extension to be essential. Because $G$ is essentially complete by
Proposition \ref{Prop:17}, we know that $G = H$. And finally, the supremum $%
\bigvee Z = h_0 \in G$ is pointwise by Proposition \ref{Prop:12}.
\end{proof}

\begin{corollary}
\label{Cor:10} $G$ is pointwise $\sigma $-complete iff it is epicomplete in $%
\mathbf{W}$.
\end{corollary}

\begin{proof}
Both conditions are equivalent to $G$ being of the form $\mathcal{R}L$ for $%
L $ a $P$-frame. One equivalence is provided by Theorem 3.4 of \cite%
{BallWaltersZenk:2010} and the other by Theorem \ref{Thm:4}.
\end{proof}

\begin{corollary}
The full subcategory comprised of the pointwise $\sigma $-complete objects
is reflective in $\mathbf{W}$. 
\end{corollary}

\begin{corollary}
The following are equivalent for an extension $G\leq H$ in $\mathbf{W}$.

\begin{enumerate}
\item 
The extension is (isomorphic to) the functorial epicompletion of $G$.

\item 
The extension is (isomorphic to) $G\rightarrow \mathcal{RP}L$, where $%
\mathcal{P}L$ designates the $P$-frame reflection of the Madden frame $L$ of 
$G$.
\end{enumerate}
\end{corollary}

\begin{proof}
A complete
description of the extension in (2) is $G \rightarrow \mathcal{RPM}G$. The
point is that the extension is the concatenation of three functors, hence it
is functorial. But there can be only one functorial epicompletion on general
grounds. See \cite{AdamekHerrlichStrecker:2004}.
\end{proof}

\end{document}